\PassOptionsToPackage{unicode}{hyperref}
\PassOptionsToPackage{hyphens}{url}
\PassOptionsToPackage{dvipsnames,svgnames,x11names}{xcolor}
\documentclass[12pt]{article}

\usepackage{amsmath,amssymb}
\usepackage{bm}
\usepackage{amsthm}

\newtheorem{assumption}{Assumption}
\newtheorem{theorem}{Theorem}
\newtheorem{remark}{Remark}
\newtheorem{proposition}{Proposition}

\newtheorem{lemma}{Lemma}

\usepackage{iftex}
\ifPDFTeX
\usepackage[T1]{fontenc}
\usepackage[utf8]{inputenc}
\usepackage{textcomp}
\else
\usepackage{unicode-math}
\defaultfontfeatures{Scale=MatchLowercase}
\defaultfontfeatures[\rmfamily]{Ligatures=TeX,Scale=1}
\let\bm\symbf
\let\boldsymbol\symbf
\fi
\usepackage{lmodern}

\IfFileExists{upquote.sty}{\usepackage{upquote}}{}
\IfFileExists{microtype.sty}{%
	\usepackage{microtype}
	\UseMicrotypeSet[protrusion]{basicmath}
}{}

\makeatletter
\@ifundefined{KOMAClassName}{%
	\IfFileExists{parskip.sty}{%
		\usepackage{parskip}
	}{%
		\setlength{\parindent}{0pt}
		\setlength{\parskip}{6pt plus 2pt minus 1pt}
	}
}{%
	\KOMAoptions{parskip=half}
}
\makeatother

\usepackage{xcolor}
\makeatletter
\ifx\paragraph\undefined\else
\let\oldparagraph\paragraph
\renewcommand{\paragraph}{
	\@ifstar
	\xxxParagraphStar
	\xxxParagraphNoStar
}
\newcommand{\xxxParagraphStar}[1]{\oldparagraph*{#1}\mbox{}}
\newcommand{\xxxParagraphNoStar}[1]{\oldparagraph{#1}\mbox{}}
\fi
\ifx\subparagraph\undefined\else
\let\oldsubparagraph\subparagraph
\renewcommand{\subparagraph}{
	\@ifstar
	\xxxSubParagraphStar
	\xxxSubParagraphNoStar
}
\newcommand{\xxxSubParagraphStar}[1]{\oldsubparagraph*{#1}\mbox{}}
\newcommand{\xxxSubParagraphNoStar}[1]{\oldsubparagraph{#1}\mbox{}}
\fi
\makeatother

\usepackage{longtable,booktabs,array}
\usepackage{calc}
\usepackage{etoolbox}
\makeatletter
\patchcmd\longtable{\par}{\if@noskipsec\mbox{}\fi\par}{}{}
\makeatother

\IfFileExists{footnotehyper.sty}{\usepackage{footnotehyper}}{\usepackage{footnote}}
\makesavenoteenv{longtable}

\usepackage{graphicx}
\makeatletter
\def\maxwidth{\ifdim\Gin@nat@width>\linewidth\linewidth\else\Gin@nat@width\fi}
\def\maxheight{\ifdim\Gin@nat@height>\textheight\textheight\else\Gin@nat@height\fi}
\makeatother
\setkeys{Gin}{width=\maxwidth,height=\maxheight,keepaspectratio}
\def\fps@figure{htbp}

\makeatletter
\@ifpackageloaded{caption}{}{\usepackage{caption}}
\AtBeginDocument{%
	\ifdefined\contentsname
	\renewcommand*\contentsname{Table of contents}
	\else
	\newcommand\contentsname{Table of contents}
	\fi
	\ifdefined\listfigurename
	\renewcommand*\listfigurename{List of Figures}
	\else
	\newcommand\listfigurename{List of Figures}
	\fi
	\ifdefined\listtablename
	\renewcommand*\listtablename{List of Tables}
	\else
	\newcommand\listtablename{List of Tables}
	\fi
	\ifdefined\figurename
	\renewcommand*\figurename{Figure}
	\else
	\newcommand\figurename{Figure}
	\fi
	\ifdefined\tablename
	\renewcommand*\tablename{Table}
	\else
	\newcommand\tablename{Table}
	\fi
}
\@ifpackageloaded{float}{}{\usepackage{float}}
\floatstyle{ruled}
\@ifundefined{c@chapter}{\newfloat{codelisting}{h}{lop}}{\newfloat{codelisting}{h}{lop}[chapter]}
\floatname{codelisting}{Listing}

\makeatother

\makeatletter
\@ifpackageloaded{subcaption}{}{\usepackage{subcaption}}
\makeatother

\ifLuaTeX
\usepackage{selnolig}
\fi

\usepackage{natbib}
\usepackage{bookmark}
\IfFileExists{xurl.sty}{\usepackage{xurl}}{}
\hypersetup{
	pdftitle={Support Recovery and $\ell_2$-Error Bound for Sparse Regression with Quadratic Measurements via Weakly-Convex-Concave Regularization},
	pdfauthor={Author 1; Author 2},
	pdfkeywords={Nonconvex statistics, Finite sample error bound, Consistency, Optimization algorithm},
	colorlinks=true,
	linkcolor={blue},
	filecolor={Maroon},
	citecolor={Blue},
	urlcolor={Blue},
	pdfcreator={LaTeX}
}

\usepackage{algorithm}
\usepackage{algcompatible}
\usepackage{makecell}
\usepackage{threeparttable}
\usepackage{siunitx}
\usepackage{multirow}
\usepackage{epstopdf}
\usepackage{rotating}
\usepackage{epsfig}  
\usepackage{fancyhdr}  
\graphicspath{{./Numerical-Experiments/random/}}   

\usepackage{hyperref}

\newcommand{\anon}{1}  

\begin{document}
	
	\def\spacingset#1{\renewcommand{\baselinestretch}{#1}\small\normalsize}
	\spacingset{1}
	
	
	\if1\anon
	{
		\title{\bf Support Recovery and $\ell_2$-Error Bound for Sparse Regression with Quadratic Measurements via Weakly-Convex-Concave Regularization}
		\author{Jun Fan, Jingyu Yang, 	Xinyu Zhang\hspace{.2cm}\\
			School of Science, Hebei University of Technolog\\
			and \\
			Liqun Wang\thanks{Corresponding author: Liqun.Wang@umanitoba.ca.
				\textit{This work was supported by the National Natural Science Foundation
					of China under Grants 12571345 and 12271022; and the Natural Sciences and Engineering Research Council of Canada under Grant 4924-2023.}}\\
			Department of Statistics, University of Manitoba}
		\maketitle
	} \fi
	
	\if0\anon
	{
		\bigskip
		\bigskip
		\bigskip
		\begin{center}
			{\LARGE\bf Title}
		\end{center}
		\medskip
	} \fi
	
	\bigskip
	\begin{abstract}
		The recovery of unknown signals from quadratic measurements finds extensive applications in fields such as phase retrieval, power system state estimation, and unlabeled distance geometry. This paper investigates the finite sample properties of weakly convex--concave regularized estimators in high-dimensional quadratic measurements models. By employing a weakly convex--concave penalized least squares approach, we establish support recovery and $\ell_2$-error bounds for the local minimizer. To solve the corresponding optimization problem, we adopt two proximal gradient strategies, where the proximal step is computed either in closed form or via a weighted $\ell_1$ approximation, depending on the regularization function. Numerical examples demonstrate the efficacy of the proposed method.
	\end{abstract}
	
	\noindent%
	{\it Keywords:} Nonconvex statistics, Finite sample error bound,  Consistency, Optimization algorithm.
	\vfill
	
	\newpage
	\spacingset{1.8} 

	\section{Introduction}\label{sec-intro}
	
Quadratic measurement regression model arises in numerous applications in physics, engineering, and data science, including phase retrieval \citep{candes2015}, generalized phase retrieval \citep{Wang2019pr}, the unassigned distance geometry problem \citep{Huang}, and power system state estimation \citep{Wangrobust}. The essence is to recover the signals from noisy quadratic measurements data. Recent years have witnessed a surge of interests and methodological developments in this area. Notable advances include \citet{thaker2020}, \citet{huang2020solving}, \citet{chen2022error}, and \citet{Fan2025}, who proposed algorithms and theoretical guarantees of optimal solutions under various problem settings. 

A particular exciting development is in the high-dimensional regime, where the dimension of the signal $d$ is large or even exceeds the number of sample points $n$. In this case, the sparsity assumption is crucial for ensuring identifiability and statistical efficiency. Several studies have explored this challenging direction, e.g., \citet{bolte2018} studied the quadratic inverse problem and proposed a Bregman proximal gradient algorithm (PGA), while \citet{zhang2023} and \citet{Ding2025} developed Bregman PGA for more general regularized problems. A special case of sparse quadratic model is sparse phase retrieval that attracted considerable  attention \citep{2019structured, xu2021sparse, cai2022sparse, pr}. More recently, \citet{chen2025} employed the thresholded Wirtinger flow algorithm of \cite{tonycai2016} for sparse phase retrieval.


However, most of the existing literature mentioned above focused primarily on numerical optimization and algorithmic development for the regularized least squares formulation of quadratic inverse problem, and the statistical guarantees remain relatively under-explored. On the other hand, in many applications the real data are noisy and in many cases, the noise level is significantly high relative to the signals. Therefore, investigation of statistical properties and stability of the methods and numerical procedures is desirable and important. A few researchers have attempted in this direction. Notably, \citet{tonycai2016} developed a thresholded Wirtinger flow algorithm for noisy sparse phase retrieval that achieves a minimax optimal rate under sub-exponential noise satisfying a near-optimal sample size condition; while \citet{Fan2018} established a weak oracle property and proposed a fixed-point iterative algorithm for the $\ell_q(0<q<1)$-regularized least squares method.

In this paper, we consider the following general quadratic measurement model
	\begin{equation}\label{qm}
		y_i = \boldsymbol{\beta}^{\top} \boldsymbol{Z}_i \boldsymbol{\beta} + \varepsilon_i, \quad i = 1, \ldots, n,
	\end{equation}
	where $y_i \in \mathbb{R}$ denotes the observed response, $\boldsymbol{Z}_i \in \mathbb{R}^{d \times d}$ is a symmetric design matrix, $\boldsymbol{\beta} \in \mathbb{R}^d$ is the vector of unknown parameters representing true signals, and $\varepsilon_i \in \mathbb{R}$ is the random noise.	
Specifically, we study the regularization problem
	\begin{equation}\label{prob}
		\min_{\bm{\beta} \in \mathbb{R}^d} ~ \frac{1}{4n} \sum_{i=1}^n \left( \bm{\beta}^{\top} \bm{Z}_i \bm{\beta} - y_i \right)^2 + P_{\lambda_n}(\bm{\beta}),
	\end{equation}
	where $\lambda_n > 0$ is a tuning parameter that encourages sparsity, and $P_{\lambda_n}(\cdot)$ is a penalty function. Model (\ref{qm}) is called generalized phase retrieval \citep{Wang2019pr, chen2022error}, whereas the phase retrieval is a special case where $\boldsymbol{Z}_i$ has rank one. This later case is also known as an index model in econometrics and statistics.


Various regularization techniques have been developed in the context of linear regression models, such as the LASSO, $\ell_q(0<q<1)$ penalty \citep{lq}, SCAD \citep{fan2001variable}, and MCP \citep{zhang2010nearly}.
It is well-known that LASSO has constant derivative that induces persistent shrinkage and biased estimation \citep{zou2006adaptive}. In contrast, concave penalties (e.g., SCAD, MCP, and $\ell_q(0<q<1)$) have derivatives that vanish for large coefficients, thereby mitigating shrinkage bias and yielding nearly unbiased estimates. However, nonconvex penalties introduces computational challenges. In particular, highly nonconvex penalties may cause numerous local minima, and the convergence of first-order algorithm (e.g., PGA) relies on favorable properties of the objective function. One such property is weak convexity, which is a key condition that generally ensures the proximal operator is single-valued under standard technical assumptions \citep{wang2010chebyshev,khanh2025inexact}. If this condition is violated, the proximal operator may be multi-valued, which poses a potential risk that iterative algorithms could fail to converge to a meaningful solution.
	
	This challenge motivates the study of regularizers that balance statistical and computational considerations. \citet{Loh2015} developed a general framework for a class of weakly convex regularized M-estimators and established finite sample error bounds between any stationary point of the penalized estimator and the population-optimal solution without regularization (i.e., the solution minimizing the expected population risk). While their framework applies to a wide range of models, including linear and generalized linear models and graphical LASSO, it cannot be directly applied to sparse quadratic measurement problems. The main difficulty is that the least squares loss for this model leads to a highly nonconvex quartic polynomial that are computationally very challenging to minimize. In fact, as pointed out by \citet{candes2015}, it is already known that determining whether a stationary point of a quartic polynomial is a local minimizer is NP-hard. Moreover, \citet{Loh2015} introduced an extra tuning parameter that must be chosen carefully to ensure that the true value $\bm{\beta}^*$ is a feasible point, thereby further increasing the difficulty of solving the problem. Weakly convex regularization methods have also been widely used in sparse signal recovery \citep{2019weakly,K2022} and linear inverse problems \citep{Sh2024,Gou2024,ebner2025}, but these works do not explicitly exploit concave penalties. 
	
	To overcome the above mentioned numerical issues, in this paper, we propose a class of concave penalties that also satisfy weak convexity, called Weakly-Convex-Concave Penalties (WCCP). The major advantage of this class is that it provides a flexible framework for analyzing local minimizers in sparse quadratic measurement problems, combining the statistical benefits of concavity with the algorithmic guarantees of weak convexity. 

	Our contributions are threefold. First, we bridge a notable gap in the literature by providing the first systematic statistical analysis of the WCCP-regularized estimator for the quadratic measurement model \eqref{qm}. We establish rigorous support recovery guarantees and $\ell_2$-error bounds for its local minimizer, and derive the statistical properties of the proposed estimator. In particular, by choosing the appropriate regularization parameter, the $\ell_2$-error bounds achieve the rate of $O(\sqrt{\ln d /n})$, which is the known optimal rate for sparse phase retrieval as shown in \citet{tonycai2016}. Second, we develop iterative algorithms tailored to different regularization functions that provide convergence guarantees. Third, we demonstrate the effectiveness of our method, particularly in the context of sparse phase retrieval, and show that it performs comparable or better than the existing LASSO-based approaches.
	
	The rest of this paper is organized as follows. Section 2 presents the assumptions and statistical properties of the estimator. Section 3 introduces two iterative algorithms and provides a convergence analysis. Numerical experiments are presented in Section 4. Technical lemmas and proofs are deferred to Appendices A and B.
	
	\subsection*{Notation}
	For a vector $\bm{v} = (v_1, \ldots, v_d)^{\top}$, let $\|\bm{v}\|$, $\|\bm{v}\|_1$, and $\|\bm{v}\|_{\infty}$ denote its Euclidean, $\ell_1$, and $\ell_{\infty}$ norms, respectively.  Denote  by $\bm{e}_{d,j}$ the $j$th column of the $d\times{}d$ identity matrix $\bm{I}_d$.	 For any $n\times{}d$ matrix $\bm{Z}$, denote $|\bm{Z}|_\infty=\max_{1\leq i\leq n,1\leq j\leq{}d}|\bm{e}_{n,i}^\top\bm{Z}\bm{e}_{d,j}|$. The submatrix of $\bm{Z}$ with rows and columns indexed by sets $\Gamma_1$ and $\Gamma_2$ is denoted $\bm{Z}^{\Gamma_1, \Gamma_2}$. The subvector of $\bm{\beta}$ indexed by $\Gamma$ is denoted $\bm{\beta}_{\Gamma}$. The $j$th standard basis vector in $\mathbb{R}^d$ is denoted $\bm{e}_{d,j}$. For a function $h(t)$, we denote its gradient and Hessian by $\nabla h(t)$ and $\nabla^2 h(t)$, respectively, and its subdifferential by $\partial h(t)$.
	Let $\bm\beta^\star$ be the true parameter value, and denote $\Gamma^\star=\mathrm{supp}(\bm\beta^\star):=\{j:{\bm{e}_{d,j}^T}\bm\beta^\star\neq 0,j=1,...,d\}$. Define the oracle regularized least square estimator as follows,
	\begin{equation*}
		\hat{\bm\beta}_1^o \in \arg\min_{\bm\beta_1\in \mathbb{R}^s}\frac{1}{4n}\sum\limits_{i=1}^n({{\bm\beta_1}^\top \bm Z_i^{\Gamma^\star\Gamma^\star}{\bm\beta_1}-y_i})^2+\sum_{j\in\Gamma^\star}p_{\lambda_n}(|\beta_j|).
	\end{equation*}

	\section{Finite Sample Statistical Results}
	In this section, we derive the finite sample properties and consistency of the weakly convex-concave estimator in the high-dimensional case where $d$ is large than $n$. We assume that $\ln d=o(n^\alpha)$ with some constant $\alpha \in (0,1)$ and denote the number of non-zero elements of the true signal by $s$.		
	Following the literature \citep{Huang2008,Fan2018}, we assume that there exists constants $0<c_1\leq c_2<\infty$ such that
	\begin{equation}\label{(3)}
		c_1 \leq\min\{|\bm{e}_{d,j}^\top\bm\beta^\star|,j\in \Gamma^\star\}\leq\max\{|\bm{e}_{d,j}^\top\bm\beta^\star|,j\in \Gamma^\star\}\leq c_2.
	\end{equation}
	Further assume that the observed response and design matrix are standardized so that
	\begin{equation}\label{Z-std}
		\sum\limits_{i=1}^ny_i=0,~~~\mbox{and}~~~\sum\limits_{i=1}^n |\bm Z_i|_{\infty}^2=n.
	\end{equation}

	For the regularization function, noise and design matrix, we make the following assumptions.		
	\begin{assumption}\label{assumption1}
		The regularization function is coordinate separable
		\begin{equation*}
			P_{\lambda_n}(\bm\beta)=\sum_{j=1}^{d}p_{\lambda_n}(|\bm\beta_j|),
		\end{equation*}
		for some scalar function $p_{\lambda _n}$ which satisfies:
		\begin{itemize}
			\item[(i)] $p_{\lambda_n}$ is concave on $(0,\infty)$ and there exists $\mu >0$ such that $p_{\lambda_n}(t)+\frac{\mu }{2}t^2$ is convex.
			\item[(ii)] $p_{\lambda_n}$ satisfies $p_{\lambda _n}(0)=0$ and $p_{\lambda_n}(t)> 0$ for $ t>0$.
			\item[(iii)] For $t>0$, $p_{\lambda_n}$ is non-decreasing and $t\mapsto\frac{p_{\lambda_n}(t)}{t}$ is non-increasing in $t$.
			\item[(iv)] \( p_{\lambda_n} \) is differentiable with derivative \( p'_{\lambda_n}\) for all \( t > 0 \), and subdifferentiable at \( t = 0 \), satisfying \( \lim_{t \to 0^+} p'_{\lambda_n}(t) = \lambda_n \varrho \) with \( \varrho > 0 \). The function \( p'_{\lambda_n} \) is locally Lipschitz continuous in \( (0,\infty) \).
		\end{itemize}
	\end{assumption}

Compared to \cite{loh2017support,Loh2017}, here we requires the penalty function to be concave and its derivative $p_{\lambda_n}'$ to be locally Lipschitz continuous on the positive half-line. We explicitly emphasize the concavity of the penalty function because it is known to possess stronger variable selection capabilities. However, we impose weak convexity here to ensure the proximal mapping of the penalty function remains single-valued, which is crucial for achieving superior algorithmic convergence. Combining the second order continuous differentiability of the loss and the local Lipschitz's continuity of $p'_{\lambda_n}$ enables us to use the generalized Hessian to verify that a stationary point is a local minimizer. In fact, many concave penalties, including SCAD and MCP, possess locally Lipschitz continuous first derivatives on the positive half-line. In the following we present some examples of regularizers satisfying Assumption 1, including SCAD, and some other commonly used regularizers.
	
{\bf Example 1.} 
The Transformed $\ell_1$ (TL1) penalty \citep{lv2009unified} is defined as
\[
p_{\lambda_n}(t) = \lambda_n \cdot \frac{(a+1)t}{a + t}, \qquad t>0,
\]
where \(a > 0\) is a shape parameter. Assumption 1 holds with \(\varrho= \dfrac{a+1}{a}\) and \(\mu= 2\lambda_n \frac{a+1}{a^2}\).
%
%

{\bf Example 2.} LOG regularizer \citep{LOG}:
\[
p_{\lambda_n}(t)=\lambda_n \log (c|t|+1),
\]
where \(\lambda_n\) is the regularization parameter and \(c>0\) is a fixed parameter. Assumption 1 holds with \(\varrho=c\) and \(\mu=\lambda_n c^2\).

{\bf Example 3.} EXP regularizer \citep{EXP}:
\[
p_{\lambda_n}(t)=\lambda_n (1-e^{-c|t|}),
\]
where \(\lambda_n\) is the regularization parameter and \(c>0\) is a fixed parameter. Assumption 1 holds with \(\varrho=c\) and \(\mu=\lambda_n c^2\).

{\bf Example 4.}
The SCAD penalty is defined as
\[
p_{\lambda}(t)=
\begin{cases}
	\lambda |t|, & |t| \leq \lambda,\\[4pt]
	-\dfrac{t^{2}-2a\lambda|t|+\lambda^{2}}{2(a-1)}, & \lambda < |t| \leq a\lambda,\\[4pt]
	\dfrac{(a+1)\lambda^{2}}{2}, & |t| > a\lambda,
\end{cases}
\]
where \(\lambda>0\) is the regularization parameter and \(a>2\) is a fixed shape parameter. Assumption 1 holds with \(\varrho=1\) and \(\mu = 1/(a-1)\).

{\bf Example 5.}
The MCP penalty takes the form
\[
p_{\lambda}(t)=
\begin{cases}
	\lambda |t| - \dfrac{t^{2}}{2b}, & |t| \leq b\lambda,\\[6pt]
	\dfrac{b\lambda^{2}}{2}, & |t| > b\lambda,
\end{cases}
\]
where \(\lambda>0\) and \(b>0\) is a shape parameter controlling the concavity. Assumption 1 holds with \(\varrho=1\) and \(\mu = 1/b\).

The weak convexity parameters for SCAD and MCP shown in Examples 4 and 5 were previously established in \cite{loh2017support} and \cite{Loh2017}.

	\begin{assumption}\label{assumption2}
		The errors $\{\varepsilon_1,\cdots,\varepsilon_n\}$ are independent and identically distributed sub-Gaussian with variance proxy $\sigma^2$, and each $\varepsilon_i$ has zero mean and positive variance.
	\end{assumption}
	
	\begin{assumption}\label{assumption3}
		For any $\bm{x},\bm{y}\in\mathbb{R}^s$, there exist constants $c_4\geq c_3>0$ satisfying
		\begin{equation*}
			c_3{\|\bm{x}\|}^{2}{\|\bm{y}\|}^{2}
			\leq \frac{1}{n}\sum\limits_{i=1}^n(\bm{x}^\top \bm Z^{\Gamma^\star\Gamma^\star}_i \bm{y})^2
			\leqslant c_4{\|\bm{x}\|}^{2}{\|\bm{y}\|}^{2}.
		\end{equation*}
	\end{assumption}

	To establish the finite‑sample properties of our proposed estimator, we first analyze the oracle estimator $\hat{\bm\beta}_{1}^o$, which is defined with prior knowledge of the true support set. The following theorem provides its non‑asymptotic consistency guarantee.

\begin{theorem}\label{theorem1}
		Under model \eqref{qm} and Assumption 1-3, if $c_4\geq3/(ns)$, then
		\begin{equation*}
			\mathbb{P}\left(\|\hat{\bm\beta}_{1}^o-\bm\beta_1^\star\|\leq r_n\right)\geq 1-p_1-p_2,
		\end{equation*}
		where $p_1=2\exp\left\{-s\ln(1+2n)\right\}, p_2=\frac{1}{n^{-n/2}}$,
		$r_n=C_0\left(\sqrt{\ln(1+2n)/n}+\lambda_n\varrho/\sqrt{s}\right)$ and $C_0=\max\left\{10\sigma\sqrt{c_4}/c_1c_3,4/c_1^2c_3\right\}$.
	\end{theorem}
	
	The condition $c_4 \geq 3/(ns)$ is not restrictive, since the right‑hand side diminishes to near zero for any practical $n$ and $s$, so that any finite uniform bound $c_4$ will automatically satisfy it. Theorem~\ref{theorem1} shows that the estimation error of the oracle estimator is bounded by $r_n$. Under the standard high‑dimensional scaling $(s\ln n + \ln d)/n \to 0$, one can choose $\lambda_n$ such that $\lambda_n/\sqrt{s} \to 0$ to imply $r_n \to 0$ and thus the consistency of $\hat{\bm\beta}_{1}^o$ as $n \to \infty$.
	
	Building on this oracle result, a standard and necessary step for analyzing the high-dimensional estimator is to impose conditions on the design matrix $\bm Z_i$ to control its column correlations.
	
	\begin{assumption}\label{assumption4}
		There exists a constant $c_5>0$ such that
		\begin{equation*}
			|\sum\limits_{i=1}^n\bm Z_i^{\Gamma^\star\Gamma^\star} \otimes \bm Z_i^{\Gamma^{\star c}\Gamma^{\star c}}|_\infty\leq c_5\sqrt{n},~|\sum\limits_{i=1}^n\bm Z_i^{\Gamma^\star\Gamma^\star}\otimes \bm Z_i^{\Gamma^{\star c}\Gamma^\star}|_\infty\leq c_5\sqrt{n},
		\end{equation*}	
		where $\otimes$ is the Kronecker product.
	\end{assumption}
	
	The inequalities in Assumption 4 are similar to the partial orthogonality condition in \cite{Huang2008} for linear models, which was later extended to quadratic models by \cite{Fan2018}. They can be regarded as that the sub-matrix of $\bm Z_i$ corresponding to $\Gamma^\star$ and the complementary matrix being orthogonal. 

	\begin{theorem}\label{theorem2}

Under model \eqref{qm} and Assumption 1-4, suppose that the triple $(n, s, d)$ satisfy
			\begin{equation}\label{ns_con}
			36c_2^2c_5C_0s^2\leq \sqrt{n}, \quad\sqrt{\frac{2\ln(1+2n)}{n}}\leq c_2C_0^{-1},\quad \frac{3}{\sqrt{ns}}\leq c_2
		\end{equation} and
		\begin{equation}\label{lambdan_con}
			\max\Big\{\sqrt{\frac{s\ln(1+2n)}{n}},6\sqrt{2}c_2\sigma s\sqrt{\frac{s\ln(1+2n)+\ln d}{n}}\Big\}\leq \lambda_n\varrho \leq\frac{c_2\sqrt{s}}{\sqrt{2}C_0}.
		\end{equation}
		Then, for any $\mu \leq c_1^2c_3s$ and
		\begin{equation}\label{rns_con}
			2(4c_3+3c_4)c_2r_n+10\sigma\sqrt{c_4}\sqrt{\frac{\ln (1+2n)}{n}}\leq c_1^2c_3\sqrt{s},
		\end{equation}
		there exists a local minimizer $\hat{\bm\beta}$ of \eqref{prob} such that
		\begin{itemize}
			\item[(i)]  $ \mathbb{P}(\hat{\bm\beta}_{{\Gamma^\star}^c}=0)\geq 1-p_1-2p_2-\frac{2}{n}$; and
			\item[(ii)] $\mathbb{P}\left(\min\{\|\hat{\bm\beta}+\bm\beta^\star\|,\|\hat{\bm\beta}-\bm\beta^\star\|\}\leq r_n\right)\geq  1-p_1-2p_2-\frac{2}{n}$.
		\end{itemize}			
	\end{theorem}

	\begin{remark}\label{rem:scaling}\emph{The conditions in Theorem~\ref{theorem2} define a non‑asymptotic feasible region for the problem dimensions $(n, s, d)$ and the regularization parameter $\lambda_n$, within which the desired statistical guarantees hold with explicit constants. From an asymptotic perspective ($n \to \infty$, allowing $s, d \to \infty$), these conditions are naturally satisfied under standard sparsity assumptions.}
		
		\emph{The inequalities in \eqref{ns_con} impose constraints on the scaling between $n$ and $s$. The first inequality, $36c_2^2c_5C_0s^2 \leq \sqrt{n}$, implies $s = O(n^{1/4})$. The remaining two inequalities involve terms that decay to zero as $n$ grows and thus are satisfied for sufficiently large $n$.}
		
		\emph{The admissible range for $\lambda_n$ in \eqref{lambdan_con} requires that the lower bound (which dominates the noise) does not exceed the upper bound (which preserves the signal). Under the standard high‑dimensional scaling $(s\ln n + \ln d)/n \to 0$ and the sparsity condition $s = O(n^{1/4})$, the lower bound is $o(1)$ while the upper bound is $O(n^{1/8})$. Hence, for large $n$, the admissible interval for $\lambda_n$ is non‑empty. Moreover, by choosing $\lambda_n$ to be of the same order as the lower bound (which tends to zero), we ensure $\lambda_n \to 0$ and $\lambda_n/\sqrt{s} \to 0$. The upper bound is only needed to control the bias in the non‑asymptotic analysis and does not conflict with $\lambda_n \to 0$ asymptotically.}
		
		\emph{The condition $\mu \leq c_1^2c_3s$ restricts the weak convexity parameter $\mu$ of the penalty function, ensuring that it does not dominate the curvature of the loss function. For typical weakly convex penalties (e.g., SCAD, MCP), $\mu$ is a fixed constant. If $s \to \infty$, the right‑hand side grows unbounded, so a fixed $\mu$ satisfies the inequality for large $n$; if $s$ remains bounded, then the condition imposes a mild fixed upper bound on $\mu$.}
		
		\emph{The inequality \eqref{rns_con} involves $r_n = C_0\bigl(\sqrt{\ln(1+2n)/n} + \lambda_n\varrho/\sqrt{s}\bigr)$. Because $\lambda_n/\sqrt{s} \to 0$ as argued above, we have $r_n \to 0$ as $n \to \infty$, which implies the estimation error bound in Theorem~\ref{theorem2} vanishes asymptotically.}
		
		\emph{In summary, these conditions jointly describe a regime where sparsity grows slowly and dimension may grow exponentially with $n$, a typical setting in high‑dimensional statistics. Therefore, all conditions in Theorem~\ref{theorem2} are compatible with both finite‑sample guarantees and asymptotic consistency.}
	\end{remark}

	\begin{remark}\emph{
			For the sparse phase retrieval problem, \citet{tonycai2016} considered parameter estimation of real-valued signals by minimizing the empirical $\ell_2$ loss function. Their results show that the estimator $\hat{\bm\beta} $ satisfies $\min\{\|\hat{\bm\beta}+\bm\beta^\star\|,\|\hat{\bm\beta}-\bm\beta^\star\|\}=O(\sqrt{\ln d /n})$ as $\ln d /n \rightarrow 0$. In our result, when the regularization parameter satisfies $\lambda_n=O(\sqrt{(s\ln n+\ln d) /n})$, the error bound simplifies to $O(\sqrt{\ln d /n})$ under the assumption that $s\ln(n)=O(\ln d)$, which shows that our estimator attains the known optimal rate for sparse phase retrieval.
			Recently, \citet{huangmeng2020} studied the estimation performance of the nonlinear Lasso of
			phase retrieval. Their results show that the estimator $\hat{\bm\beta}$ satisfies $\min\{\|\hat{\bm\beta}+\bm\beta^\star\|,\|\hat{\bm\beta}-\bm\beta^\star\|\}\leq \|\varepsilon\|/\sqrt{n}$. According to \citet{cai2009recovery}, the error $\bm\varepsilon$ satisfies $\|\bm\varepsilon\|\leq \sigma\sqrt{n+2\sqrt{n \ln n}}$ with high probability as $n\rightarrow \infty$. In this case, the error bound simplifies to $\sigma\sqrt{1+2\sqrt{n\ln n}/n}$. It is clear that the estimation error exceeds the estimation error bound established in our work.}
	\end{remark}

	\section{Optimization Algorithm}
	In this section, we discuss the numerical computation of the problem to assess the performance of the proposed method.  Since $n$ and $d$ are fixed, we drop them in $\lambda_n=\lambda$ and write (\ref{prob}) as  
	\begin{equation}\label{prob-opt}
		\min_{\beta\in\mathbb{R}^p} F(\bm{\beta}):= L(\bm\beta)+P_{\lambda}(\bm{\beta}),
	\end{equation}
	where $L(\bm{\beta}):=\frac{1}{4n} \sum_{i=1}^n ( \bm{\beta}^{\top} \bm{Z}_i \bm{\beta} - y_i )^2$ and $\lambda>0$. Since $L(\bm{\beta})$ is smooth and 	$P_{\lambda}(\bm{\beta})$ is weakly convex-concave, we employ the PGA to solve the problem (\ref{prob-opt}). The PGA is an established classical algorithm \citep{beck2017first} and has been applied in various inverse and optimization problems (e.g., \cite{bolte2018}, \citet{2019structured},\citet{zhang2023}, \cite{fan2025robust}). The core component of PGA is the proximal operator. Recall that, given a function $P: \mathbb{R}^d\to (-\infty, \infty)$, the proximal operator associated with $P$ is defined as
	\begin{equation*}
		\mathrm{prox}_{P}(\bm{v})
		= \arg\min_{\bm{u}}\left\{
		\frac{1}{2}\|\bm{u}-\bm{v}\|^2 + P(\bm{u})
		\right\}.
	\end{equation*}
	In particular, given a weight vector $\bm{w} = (w_1, \dots, w_d)^\mathsf{T}$ with $w_j > 0$ for all $j$, we denote the proximal operator of the weighted $\ell_1$-norm $\|\bm{w} \circ \bm{u}\|_1$ at $\bm{v}$ by
	$$
	\mathcal{S}(\bm{v}, \bm{w}) := \bm{v} - \max\{ -\bm{w}, \min\{\bm{v}, \bm{w} \} \},
	$$
	where the $\max$ and $\min$ operations are performed component-wise.  
	Note that this reduces to the classical soft thresholding operator $\mathrm{prox}_{\lambda\|\cdot\|_1}(\bm{v})$ when $w_j = \lambda$ for all $j$.
	For proximal operators that does not admit a closed-form solution, we can adopt weighted $\ell_1$ algorithm based on the concavity of the regularization function.  
	
	This analysis yields two alternative fixed-point characterizations of the minimizers as follows.

	\begin{proposition}\label{proposition 1}
		There exists a constant $r>0$, such that for any $\tau \in (0,\hat{L}^{-1}]$ and minimizer $\hat{\bm\beta}$ of problem \eqref{prob}, it holds
		\begin{equation}\label{fixed1}
			\hat{\bm\beta} = \operatorname{prox}_{\tau P_{\lambda}} \left( \hat{\bm\beta} - \tau \nabla L(\hat{\bm\beta}) \right)
		\end{equation}
		and
		\begin{equation}\label{fixed2}
			\hat{\bm\beta}=\mathcal{S}(\hat{\bm\beta}-\tau\nabla L(\hat{\bm\beta}),\tau \hat{\bm{w}}_\lambda),
		\end{equation}
where $\hat{L}=\sup_{\|\bm\beta\|\in B_r}\|\nabla^2L(\bm\beta)\|$, $B_r:=\left\{\bm\beta \in \mathbb{R}^d :\|\bm\beta\|\leq r\right\}$ and $\hat{\bm{w}}_\lambda$ is a vector with $j$th element
		\begin{equation*}
			\bm{e}_{d,j}^\top\hat{\bm{w}}_\lambda=\begin{cases}
				\max\big(p_{\lambda}^{'}(|\hat{\beta}_j|),\epsilon_1\big), & \text{if } \hat{\beta}_j\neq 0,\\
				\lambda \varrho, & \text{if } \hat{\beta}_j=0.
			\end{cases}
		\end{equation*}
		Here $\epsilon_1$ is a positive number.
	\end{proposition}		
	
	The above two characterizations provide the foundation for our algorithmic design. While the iterative scheme for solving \eqref{prob} is outlined in Algorithm 1 according to (\ref{fixed1}), the weight  $\hat{\bm{w}}_\lambda$ in (\ref{fixed2}) depends on the optimal solution $\hat{\bm{\beta}}$ which is unknown. To overcome this issue, we adopt the  framework of iteratively reweighted $\ell_1 $ algorithms (\cite{zou2006adaptive},\citealp{bai2024avoiding}). Within this framework, we employ the Majorization-Minimization (MM) technique, where, at each iteration, the regularization term is linearly approximated via its first-order expansion, which constructs a surrogate function. This leads to a sequence of sub-problems of the form
	\begin{equation*}
		\min_{\bm\beta \in \mathbb{R}^d} L(\bm\beta) + \|\bm{w}_\lambda^k\circ\bm{\beta}\|_1,
	\end{equation*}
	where \( \bm{w}_\lambda^k  \) is adaptively updated weights computed from the current iterate \( \bm\beta^{k} \). The next iterate 
	$\bm\beta^{k+1}$ is then updated according to the rule specified in Algorithm 2, which stems from an approximate minimization step for this surrogate function.
	
	\begin{algorithm}[h]
		\caption{}
		\label{alg:algorithm1}
		\begin{algorithmic}[1]
			\REQUIRE Data $\{y_i,\bm{Z}_i\}$, parameters $T, \epsilon>0, \delta>0, \gamma_0\in(0,1), \gamma_1\in(0,1)$
			\ENSURE $\bm{\beta}^{k+1}$
			\STATE \textbf{Initialization:} Choose a spectral initialization point $\bm\beta^0$ and set $k= 0$
			\STATE \textbf{Computation:} Compute $\bm\beta^{k+1}=\text{prox}_{\tau_k P_{\lambda}}(\bm\beta ^k-\tau_k \nabla L(\bm\beta^k))$, where $\tau_k=\gamma_1 \gamma_0^{j_k}$ and $j_k$ is the smallest nonnegative integer such that:
			\vspace{-1em}
			\begin{equation}\label{armijo1}
				F(\bm\beta^k)-F(\bm\beta^{k+1})\geq \delta \|\bm\beta^{k+1}-\bm\beta ^k\|^2.
			\end{equation}
			\vspace{-1em}
			\WHILE{$\|\bm \beta^{k+1}-\bm\beta^k\| \geq  \epsilon \max \{1,\|\bm\beta^k\|\}$ \textbf{and} $k\leq T$}
			\STATE Set $k=k+1$, go back to Step 3
			\ENDWHILE
			\STATE \textbf{Output:} $\bm{\beta}^{k+1}$
		\end{algorithmic}
	\end{algorithm}

	\begin{algorithm}[h]
		\caption{}
		\label{alg:algorithm2}
		\begin{algorithmic}[1]
			\REQUIRE Data $\{y_i,\bm{Z}_i\}$, parameters $T, \epsilon,\epsilon_1>0, \delta>0, \gamma_0\in(0,1), \gamma_1\in(0,1)$
			\ENSURE $\bm{\beta}^{k+1}$
			
			\STATE \textbf{Initialization:} Choose a spectral initialization point $\bm\beta^0$ and set $k= 0$
			
			\STATE \textbf{Weight computation:} Take $\bm w_\lambda^k=(w_1^k,\cdots,w_d^k)^\top$ with $w_{\lambda,j}^k=\max(p_{\lambda}^{'}(|\beta_j^k|),\epsilon_1)$ if $\beta_j^k\neq 0$ and $w_j^k=\lambda\varrho$ if $\beta_j^k=0$
			
			\STATE \textbf{Update:} Compute $\bm\beta^{k+1}=S(\bm\beta ^k-\tau_k \nabla L(\bm\beta^k),\lambda\tau_k\bm w^k)$, where $\tau_k=\gamma \alpha^{j_k}$ and $j_k$ is the smallest nonnegative integer such that condition (\ref{armijo1}) holds
			
			\WHILE{$\|\bm\beta^{k+1}-\bm\beta^k\| \geq  \varepsilon \max \{1,\|\bm\beta^k\|\}$ \textbf{and} $k\leq T$}
			\STATE Set $k=k+1$, go back to Step 3
			\ENDWHILE
			
			\STATE \textbf{Output:} $\bm{\beta}^{k+1}$
		\end{algorithmic}
	\end{algorithm}

Note that computing the step size $\tau_k$ is crucial, and its selection depends on the parameter $\hat L$ that is difficult to determine. To address this issue, we employ the Armijo line search method (\ref{armijo1}). The existence of a smallest $j_k$ and the admissible range for  $\tau_k$ are established in Lemma \ref{armijo} of Appendix B.
Having established the step size selection strategy,  we now turn to the convergence analysis of the algorithms.
	
	\begin{proposition}\label{alg-convergence}
		Let $\left\{\bm\beta^k\right\}$ be the sequence generated by Algorithm 1 or 2. 
		\begin{itemize}
			\item[(i)] $\left\{F(\bm\beta^k)\right\}$ and $\left\{F^k(\bm\beta^k)\right\}$ both are monotonically decreasing.
			\item[(ii)] $\lim\limits_{k\rightarrow \infty}\|\bm\beta^{k+1}-\bm\beta^k\|=0.$
			\item[(iii)] Every accumulation point of $\{\bm\beta^k\}$ generated by Algorithm 1 satisfies the fixed-point equation \eqref{fixed1}, and that generated by Algorithm 2 satisfies \eqref{fixed2}.
			
		\end{itemize}
	\end{proposition}

	\section{Numerical Experiments}
	
	In this section, we examine the finite sample performance of the proposed method through numerical simulations. All experiments are implemented using MATLAB(R2023b) and the code is available at \url{https://github.com/fjmath/sparseQMR_WCCP}.

We compare the performance of several WCCP, including SCAD, MCP, TL1, LOG, and EXP regularizers, with $\ell_1$ and $\ell_{1/2}$ regularizers.  
To enable a fair comparison across penalties where $\mu$ scales differently with $\lambda_n$, we introduce the standardized concavity
\[
\kappa :=\mu/\lambda_n,
\]
a dimensionless constant determined solely by each penalty's shape parameters (Table~\ref{tab:penalty_parameters}). This quantity captures the intrinsic curvature of the penalty and aligns with the maximum concavity measure of \citet{zhang2010nearly}, expressed here within the $\mu$-amenable framework of \citet{loh2017support}.

\begin{table}[htbp]
	\centering
	\caption{Shape parameters and standardized concavity \(\kappa = \mu/\lambda_n\) for the seven penalties, ordered by increasing concavity.}
	\label{tab:penalty_parameters}
	\begin{tabular}{lccccccc}
		\toprule
		Penalty & $\ell_1$ & MCP & SCAD & TL1 & EXP & LOG & $\ell_{1/2}$ \\
		\midrule
		Shape param. & --- & \(\gamma=3\) & \(a=3.7\) & \(a=3\) & \(\sigma=0.5\) & \(\varepsilon=0.1\) & \(q=0.5\) \\
		\(\kappa\) & \(0\) & \(\frac{1}{\gamma}\) & \(\frac{1}{a-1}\) & \(\frac{2(a+1)}{a^2}\) & \(\frac{1}{\sigma^2}\) & \(\frac{1}{\varepsilon^2}\) & \(\infty\) \\
		\bottomrule
	\end{tabular}
\end{table}

In Table \ref{tab:penalty_parameters} the shape parameters are set to commonly used values in the literature. The penalties are ordered by increasing standardized concavity \(\kappa = \mu/\lambda_n\), from the convex$\ell_1$ (\(\kappa=0\)) to the extremely concave $\ell_{1/2}$ (\(\kappa=\infty\)). A larger \(\kappa\) indicates a stronger concavity, leading to more aggressive sparsity promotion but also increased optimization difficulty. The corresponding numerical values are: \(\frac{1}{\gamma}\approx0.333\), \(\frac{1}{a-1}\approx0.370\), \(\frac{2(a+1)}{a^2}=\frac{8}{9}\approx0.889\), \(\frac{1}{\sigma^2}=4\), \(\frac{1}{\varepsilon^2}=100\).

The initial point is obtained using the sparse spectral initialization method proposed by \cite{chen2025}. For all regularization methods, the optimal step size is determined using the Armijo line search. The regularization parameter is chosen according to the scheme of \cite{chen2022error}, defined as $\lambda_n=\sqrt{\frac{c \ln (d)}{n^2}\sum_{i=1}^{n}(\bm\beta^\top\bm Z_i\bm\beta-y_i)^2}\|\bm\beta\|$, where the coefficient $c$ is determined via cross-validation. The relative error in the experiments is computed as follows,
\begin{equation*}
	\text{RelErr}=\frac{\min \left\{\|\hat{\bm\beta}-\bm\beta^\star\|,\|\hat{\bm\beta}+\bm\beta^\star\|\right\}}{\|\bm\beta^\star\|}.
\end{equation*}
Performance is then evaluated by
\begin{itemize}
	\item {Success rate}: proportion of trials with \(\text{RelErr} < 10^{-3}\);
	\item {Support recovery}: true positive rate (TPR), false positive rate (FPR), and F1 score;
	\item {Bias for large coefficients}: average relative deviation \(|(\hat{x}_j-x_j)/x_j|\) over the large‑coefficient support;
	\item {Computational time}.
\end{itemize}

\subsection{Experiment 1: Synthetic quadratic measurements}

We first study recovery performance using randomly generated quadratic measurements. The sensing matrices $\bm Z_i$ are generated as symmetric matrices with Gaussian entries, and the noise vector $\bm\varepsilon$ follows a sub-Gaussian distribution. The true signal $\bm\beta^\star$ is generated as a sparse Gaussian vector.

Additive noise $\varepsilon_i\sim\sigma\mathcal N(0,1)$ is added to the measurements $y_i$. Two problem settings are considered $
(d,s)=(128,10)$ and $(d,s)=(256,15).$
The sampling ratios are chosen as
$
n/d\in\{0.1,0.2,\dots,1.0\},
$
and each configuration is repeated 100 times. Recovery is regarded as successful when $\text{RelErr}<10^{-3}$.

Figure~\ref{fig_exa1} reports the success rates as functions of the sampling ratio $n/d$. Across all experimental settings, WCCP substantially outperform the convex $\ell_1$ regularization in the moderate undersampling regime ($0.3\le n/d\le0.6$). For example, when $n/d=0.4$ with $d=128$, $s=10$, and $\sigma=0.01$, the success rates of WCCP range from $0.30$ (MCP) to $0.72$ (TL1), while the $\ell_1$ method barely exceeds $0.01$.

Among the WCCP methods, TL1 consistently achieves the highest success rates and often reaches $0.9$ once $n/d\approx0.5$ under low noise. LOG and EXP also exhibit strong performance and closely follow TL1. The $\ell_{1/2}$ penalty shows mixed behavior relative to $\ell_1$: although it sometimes attains slightly higher F1 scores due to lower false positive rates, its success rate and estimation accuracy remain inferior to those of the WCCP methods.

Detailed statistics (mean and standard deviation) for RelErr, TPR, FPR, F1 score, bias for large coefficients, estimated sparsity, and computational time are reported in Appendix~C (Tables~6–21). The results indicate that WCCP methods consistently achieve lower estimation error, improved support recovery, and reduced shrinkage bias compared with $\ell_1$, particularly when the sampling ratio is moderate.

These results demonstrate that weakly convex concave penalties significantly improve sparse recovery under quadratic measurements, especially in undersampled regimes. In particular, TL1 and LOG provide the most favorable balance between recovery accuracy and sparsity control.

\begin{figure}[thpb]
	\centering	
	\begin{subfigure}{0.4\textwidth}
		\centering
		\includegraphics[width=\linewidth]{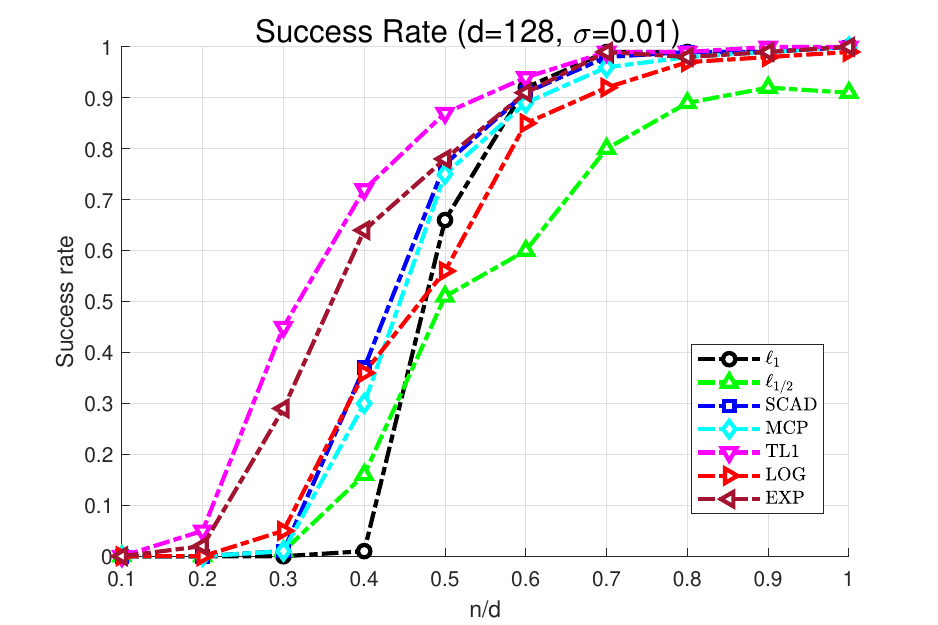}
	\end{subfigure}
	\hfill
	\begin{subfigure}{0.4\textwidth}
		\centering
		\includegraphics[width=\linewidth]{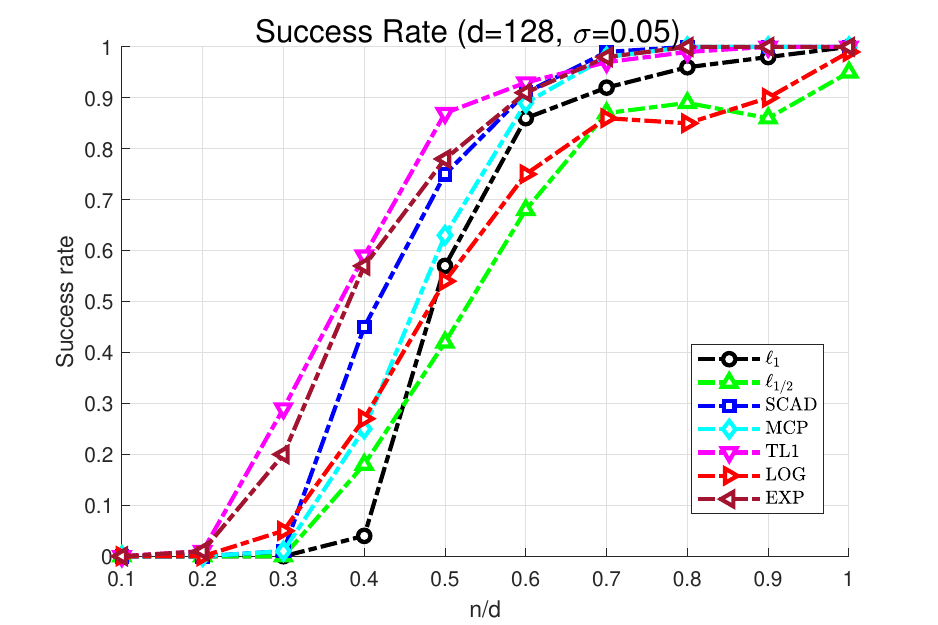}
	\end{subfigure}	
	\vspace{0.5mm}	
	\begin{subfigure}{0.4\textwidth}
		\centering
		\includegraphics[width=\linewidth]{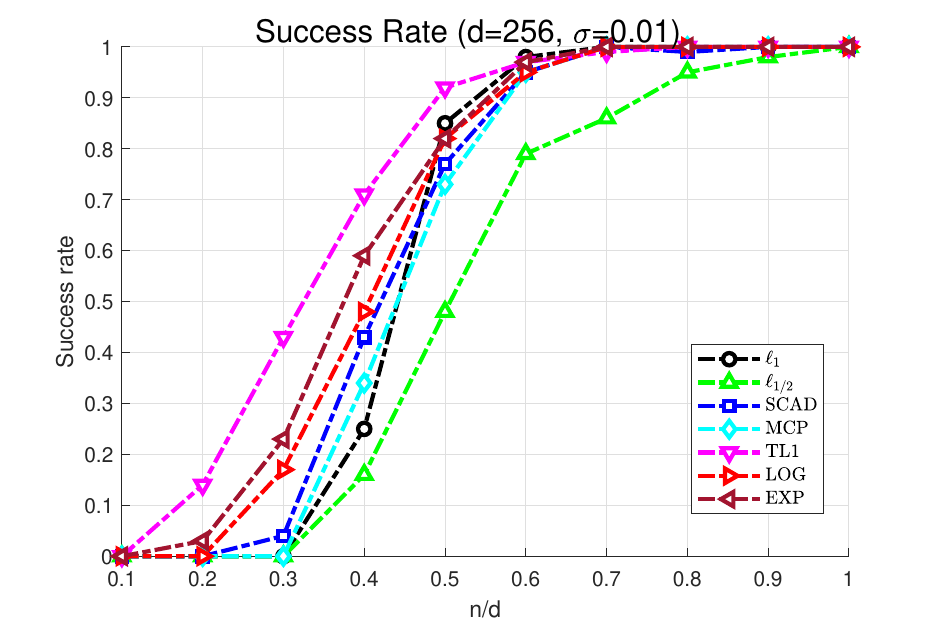}
	\end{subfigure}
	\hfill
	\begin{subfigure}{0.4\textwidth}
		\centering
		\includegraphics[width=\linewidth]{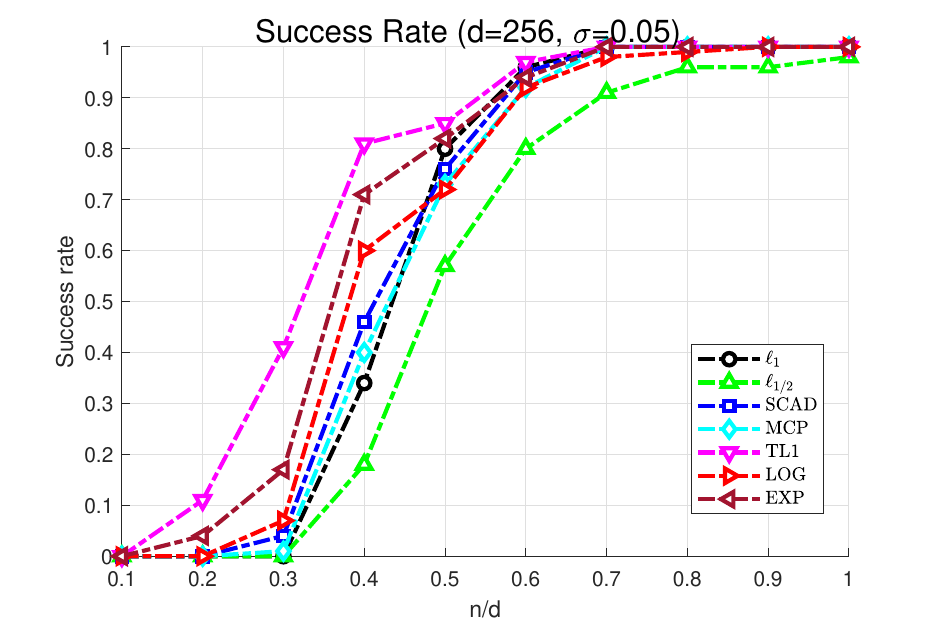}
	\end{subfigure}	
	\caption{ }
	\label{fig_exa1}
\end{figure}

\subsection{Experiment 2: Sparse phase retrieval}

The second experiment evaluates the proposed approach on sparse phase retrieval. Because the proximal operator of the EXP regularizer is computationally expensive, Algorithm~2 is used to solve the corresponding optimization problem.

The sparse signal is generated from the standard test image \texttt{cameraman.tif}. Following the wavelet-based sparsification procedure of \citet{shah2021sparse}, the image is resized to $64\times64$ pixels and transformed using a four-level Haar wavelet transform. Only the largest $5\%$ of coefficients are retained to form the sparse representation. This differs from \citet{shah2021sparse}, which retains a fixed number of coefficients.

Measurement noise is generated as $\varepsilon_i\sim\sigma\mathcal N(0,1)$.

Tables~\ref{tab:ssim_noise} and \ref{tab:ssim_results} report reconstruction quality measured by the structural similarity index (SSIM). Several WCCR penalties, including SCAD, MCP, TL1, and EXP, consistently achieve near-perfect reconstruction (SSIM $\approx0.9992$) across a wide range of measurement ratios and noise levels. In contrast, the classical $\ell_1$ regularization shows noticeable degradation when the number of measurements is limited. For example, at $n/d=0.4$, MCP already attains nearly perfect reconstruction, while $\ell_1$ achieves only SSIM $=0.708$.

The $\ell_{1/2}$ penalty is computationally efficient but appears less stable when the number of measurements is small, where reconstruction quality exhibits larger variability.

Tables~\ref{tab:time_noise} and \ref{tab:time_results} summarize the computational time. SCAD and MCP achieve competitive or faster runtimes than $\ell_1$ while maintaining superior reconstruction accuracy. For instance, at $n/d=0.2$, SCAD and MCP require $0.62$ and $0.38$ seconds respectively, compared with $0.97$ seconds for $\ell_1$. Although TL1 also achieves near-perfect reconstruction, its computational cost is substantially higher.

These results confirm that the advantages of WCCP extend to practical imaging problems. In particular, SCAD and MCP achieve highly accurate reconstructions with competitive computational cost, whereas TL1 attains excellent reconstruction quality at the expense of increased computation.

\begin{figure}[htpb]
	\centering
	\includegraphics[width=\linewidth]{"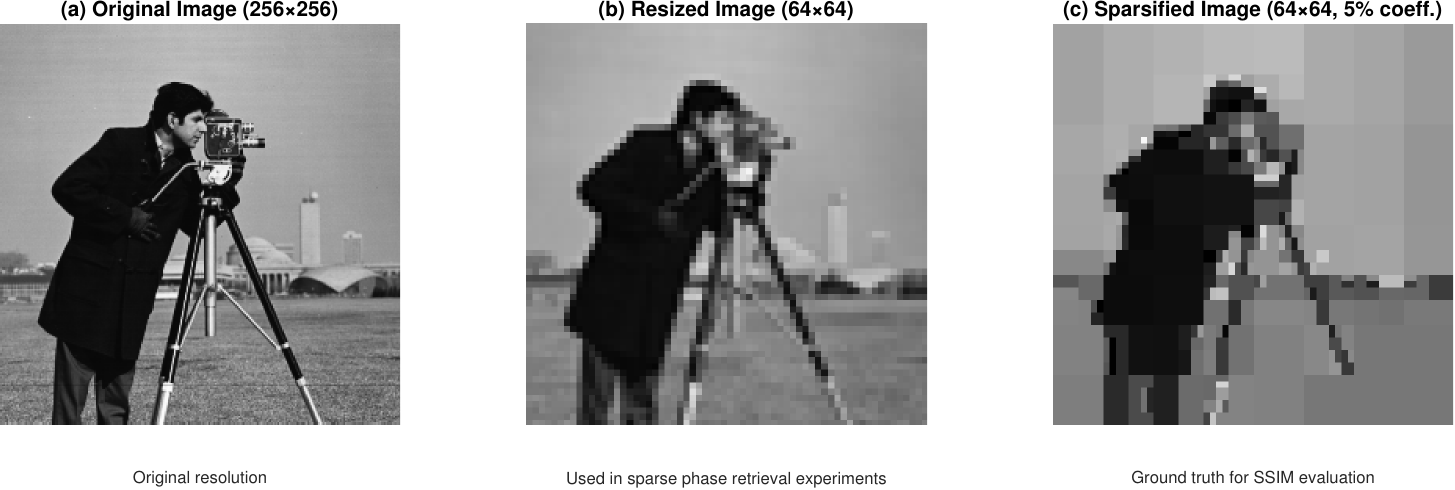"}
	\caption{}
	\label{cameraman}
\end{figure}

\begin{table}[htbp]
	\centering
	\caption{SSIM performance of sparse regularization methods at different noise levels (measurement ratio $n/d = 0.50$)}
	\label{tab:ssim_noise}
	\begin{threeparttable}
		\scriptsize
		\renewcommand{\arraystretch}{1.4}
		\begin{tabular}{@{}l *{6}{c}@{}}
			\toprule
			\multirow{2}{*}{Method} & \multicolumn{6}{c}{Noise Level $\sigma$} \\
			\cmidrule(l){2-7}
			& {0} & {$10^{-5}$} & {$10^{-4}$} & {$10^{-3}$} & {$10^{-2}$} & {$10^{-1}$} \\
			\midrule
			$\ell_1$    & \shortstack{0.9250\\{\tiny (0.2195)}} 
			& \shortstack{0.8525\\{\tiny (0.2946)}} 
			& \shortstack{0.8248\\{\tiny (0.3049)}} 
			& \shortstack{0.9750\\{\tiny (0.1008)}} 
			& \shortstack{0.9256\\{\tiny (0.2216)}} 
			& \shortstack{0.7986\\{\tiny (0.3106)}} \\
			$\ell_{1/2}$ & \shortstack{0.9992\\{\tiny (0.0000)}} 
			& \shortstack{0.9655\\{\tiny (0.1505)}} 
			& \shortstack{0.9205\\{\tiny (0.2422)}} 
			& \shortstack{0.9614\\{\tiny (0.1687)}} 
			& \shortstack{0.9992\\{\tiny (0.0000)}} 
			& \shortstack{0.9212\\{\tiny (0.2402)}} \\
			SCAD  & \shortstack{0.9991\\{\tiny (0.0000)}} 
			& \shortstack{0.9575\\{\tiny (0.1862)}} 
			& \shortstack{0.9991\\{\tiny (0.0002)}} 
			& \shortstack{0.9607\\{\tiny (0.1720)}} 
			& \shortstack{0.9992\\{\tiny (0.0000)}} 
			& \shortstack{0.9991\\{\tiny (0.0001)}} \\
			MCP   & \shortstack{0.9992\\{\tiny (0.0000)}} 
			& \shortstack{0.9581\\{\tiny (0.1836)}} 
			& \shortstack{0.9992\\{\tiny (0.0000)}} 
			& \shortstack{0.9992\\{\tiny (0.0000)}} 
			& \shortstack{0.9992\\{\tiny (0.0000)}} 
			& \shortstack{0.9992\\{\tiny (0.0000)}} \\
			TL1   & \shortstack{0.9992\\{\tiny (0.0000)}} 
			& \shortstack{0.9992\\{\tiny (0.0000)}} 
			& \shortstack{0.9624\\{\tiny (0.1645)}} 
			& \shortstack{0.9992\\{\tiny (0.0000)}} 
			& \shortstack{0.9992\\{\tiny (0.0000)}} 
			& \shortstack{0.9992\\{\tiny (0.0000)}} \\
			LOG   & \shortstack{0.9605\\{\tiny (0.1731)}} 
			& \shortstack{0.8780\\{\tiny (0.2961)}} 
			& \shortstack{0.8429\\{\tiny (0.3208)}} 
			& \shortstack{0.9251\\{\tiny (0.2279)}} 
			& \shortstack{0.9227\\{\tiny (0.2355)}} 
			& \shortstack{0.8830\\{\tiny (0.2840)}} \\
			EXP   & \shortstack{0.9992\\{\tiny (0.0000)}} 
			& \shortstack{0.9992\\{\tiny (0.0000)}} 
			& \shortstack{0.9258\\{\tiny (0.2259)}} 
			& \shortstack{0.9992\\{\tiny (0.0000)}} 
			& \shortstack{0.9992\\{\tiny (0.0000)}} 
			& \shortstack{0.9608\\{\tiny (0.1716)}} \\
			\bottomrule
		\end{tabular}
	\end{threeparttable}
\end{table}

\begin{table}[htbp]
	\centering
	\caption{Computational time of sparse regularization methods at different noise levels (measurement ratio $n/d = 0.50$)}
	\label{tab:time_noise}
	\begin{threeparttable}
		\scriptsize
		\renewcommand{\arraystretch}{1.4}
		\begin{tabular}{@{}l *{6}{c}@{}}
			\toprule
			\multirow{2}{*}{Method} & \multicolumn{6}{c}{Noise Level $\sigma$} \\
			\cmidrule(l){2-7}
			& {0} & {$10^{-5}$} & {$10^{-4}$} & {$10^{-3}$} & {$10^{-2}$} & {$10^{-1}$} \\
			\midrule
			$\ell_1$    & \shortstack{2.5091\\{\tiny (0.0971)}} 
			& \shortstack{2.8332\\{\tiny (0.5230)}} 
			& \shortstack{2.7009\\{\tiny (0.4100)}} 
			& \shortstack{3.4049\\{\tiny (0.3444)}} 
			& \shortstack{2.4327\\{\tiny (0.0431)}} 
			& \shortstack{2.4057\\{\tiny (0.0767)}} \\
			$\ell_{1/2}$ & \shortstack{0.7674\\{\tiny (0.1992)}} 
			& \shortstack{0.9997\\{\tiny (0.5615)}} 
			& \shortstack{1.2502\\{\tiny (0.8086)}} 
			& \shortstack{1.1966\\{\tiny (0.7994)}} 
			& \shortstack{0.7921\\{\tiny (0.3556)}} 
			& \shortstack{0.8601\\{\tiny (0.5791)}} \\
			SCAD  & \shortstack{2.5128\\{\tiny (0.1457)}} 
			& \shortstack{2.7934\\{\tiny (0.5774)}} 
			& \shortstack{2.7956\\{\tiny (0.4786)}} 
			& \shortstack{3.3247\\{\tiny (0.4467)}} 
			& \shortstack{2.4365\\{\tiny (0.1589)}} 
			& \shortstack{2.3920\\{\tiny (0.1716)}} \\
			MCP   & \shortstack{1.9420\\{\tiny (0.3407)}} 
			& \shortstack{2.3035\\{\tiny (0.7515)}} 
			& \shortstack{2.3677\\{\tiny (0.5875)}} 
			& \shortstack{2.5085\\{\tiny (0.5659)}} 
			& \shortstack{1.7920\\{\tiny (0.3211)}} 
			& \shortstack{1.6350\\{\tiny (0.4330)}} \\
			TL1   & \shortstack{7.0248\\{\tiny (1.0057)}} 
			& \shortstack{7.5769\\{\tiny (1.1488)}} 
			& \shortstack{7.0298\\{\tiny (1.6241)}} 
			& \shortstack{7.0558\\{\tiny (1.0862)}} 
			& \shortstack{6.6681\\{\tiny (1.0291)}} 
			& \shortstack{6.6241\\{\tiny (1.0589)}} \\
			LOG   & \shortstack{0.7942\\{\tiny (0.4062)}} 
			& \shortstack{0.9953\\{\tiny (0.6267)}} 
			& \shortstack{1.2440\\{\tiny (0.9794)}} 
			& \shortstack{1.1928\\{\tiny (0.8394)}} 
			& \shortstack{0.8852\\{\tiny (0.5659)}} 
			& \shortstack{0.8874\\{\tiny (0.6563)}} \\
			EXP   & \shortstack{1.7747\\{\tiny (0.2575)}} 
			& \shortstack{2.1511\\{\tiny (0.5463)}} 
			& \shortstack{2.0821\\{\tiny (0.4885)}} 
			& \shortstack{2.2267\\{\tiny (0.4384)}} 
			& \shortstack{1.6534\\{\tiny (0.2419)}} 
			& \shortstack{1.5930\\{\tiny (0.2972)}} \\
			\bottomrule
		\end{tabular}
	\end{threeparttable}
\end{table}

\begin{table}[htbp]
	\centering
	\caption{SSIM performance at different measurement ratios (noise level $\sigma=10^{-2}$)}
	\label{tab:ssim_results}
	\begin{threeparttable}
		\scriptsize
		\renewcommand{\arraystretch}{1.4}
		\begin{tabular}{@{}l *{10}{c}@{}}
			\toprule
			\multirow{2}{*}{Method} & \multicolumn{10}{c}{Measurement Ratio $(n/d)$} \\
			\cmidrule(l){2-11}
			& {0.1} & {0.2} & {0.3} & {0.4} & {0.5} & {0.6} & {0.7} & {0.8} & {0.9} & {1.0} \\
			\midrule
			$\ell_1$    & \shortstack{0.0320\\{\tiny (0.0106)}} 
			& \shortstack{0.0515\\{\tiny (0.0220)}} 
			& \shortstack{0.2598\\{\tiny (0.1963)}} 
			& \shortstack{0.7080\\{\tiny (0.3404)}} 
			& \shortstack{0.9639\\{\tiny (0.1494)}} 
			& \shortstack{0.9348\\{\tiny (0.1973)}} 
			& \shortstack{0.8587\\{\tiny (0.2892)}} 
			& \shortstack{0.9992\\{\tiny (0.0000)}} 
			& \shortstack{0.9992\\{\tiny (0.0000)}} 
			& \shortstack{0.9992\\{\tiny (0.0000)}} \\
			$\ell_{1/2}$ & \shortstack{0.0254\\{\tiny (0.0115)}} 
			& \shortstack{0.0416\\{\tiny (0.0233)}} 
			& \shortstack{0.2434\\{\tiny (0.3296)}} 
			& \shortstack{0.7720\\{\tiny (0.3599)}} 
			& \shortstack{0.9309\\{\tiny (0.2101)}} 
			& \shortstack{0.9992\\{\tiny (0.0000)}} 
			& \shortstack{0.9992\\{\tiny (0.0000)}} 
			& \shortstack{0.9992\\{\tiny (0.0000)}} 
			& \shortstack{0.9992\\{\tiny (0.0000)}} 
			& \shortstack{0.9992\\{\tiny (0.0000)}} \\
			SCAD  & \shortstack{0.0252\\{\tiny (0.0130)}} 
			& \shortstack{0.0478\\{\tiny (0.0302)}} 
			& \shortstack{0.4193\\{\tiny (0.2743)}} 
			& \shortstack{0.9983\\{\tiny (0.0023)}} 
			& \shortstack{0.9596\\{\tiny (0.1768)}} 
			& \shortstack{0.9992\\{\tiny (0.0000)}} 
			& \shortstack{0.9992\\{\tiny (0.0000)}} 
			& \shortstack{0.9992\\{\tiny (0.0000)}} 
			& \shortstack{0.9992\\{\tiny (0.0000)}} 
			& \shortstack{0.9992\\{\tiny (0.0000)}} \\
			MCP   & \shortstack{0.0355\\{\tiny (0.0097)}} 
			& \shortstack{0.1307\\{\tiny (0.1385)}} 
			& \shortstack{0.6638\\{\tiny (0.2226)}} 
			& \shortstack{0.9992\\{\tiny (0.0000)}} 
			& \shortstack{0.9992\\{\tiny (0.0000)}} 
			& \shortstack{0.9622\\{\tiny (0.1652)}} 
			& \shortstack{0.9992\\{\tiny (0.0000)}} 
			& \shortstack{0.9992\\{\tiny (0.0000)}} 
			& \shortstack{0.9992\\{\tiny (0.0000)}} 
			& \shortstack{0.9992\\{\tiny (0.0000)}} \\
			TL1   & \shortstack{0.0370\\{\tiny (0.0126)}} 
			& \shortstack{0.1362\\{\tiny (0.1645)}} 
			& \shortstack{0.7938\\{\tiny (0.3010)}} 
			& \shortstack{0.9992\\{\tiny (0.0000)}} 
			& \shortstack{0.9992\\{\tiny (0.0000)}} 
			& \shortstack{0.9992\\{\tiny (0.0000)}} 
			& \shortstack{0.9622\\{\tiny (0.1653)}} 
			& \shortstack{0.9992\\{\tiny (0.0000)}} 
			& \shortstack{0.9992\\{\tiny (0.0000)}} 
			& \shortstack{0.9992\\{\tiny (0.0000)}} \\
			LOG   & \shortstack{0.0352\\{\tiny (0.0184)}} 
			& \shortstack{0.0466\\{\tiny (0.0238)}} 
			& \shortstack{0.3507\\{\tiny (0.2983)}} 
			& \shortstack{0.6856\\{\tiny (0.3864)}} 
			& \shortstack{0.9620\\{\tiny (0.1660)}} 
			& \shortstack{0.9613\\{\tiny (0.1692)}} 
			& \shortstack{0.9992\\{\tiny (0.0000)}} 
			& \shortstack{0.9992\\{\tiny (0.0000)}} 
			& \shortstack{0.9992\\{\tiny (0.0000)}} 
			& \shortstack{0.9992\\{\tiny (0.0000)}} \\
			EXP   & \shortstack{0.0300\\{\tiny (0.0106)}} 
			& \shortstack{0.0460\\{\tiny (0.0360)}} 
			& \shortstack{0.4602\\{\tiny (0.4155)}} 
			& \shortstack{0.8466\\{\tiny (0.3147)}} 
			& \shortstack{0.9649\\{\tiny (0.1533)}} 
			& \shortstack{0.9992\\{\tiny (0.0000)}} 
			& \shortstack{0.9290\\{\tiny (0.2159)}} 
			& \shortstack{0.9992\\{\tiny (0.0000)}} 
			& \shortstack{0.9992\\{\tiny (0.0000)}} 
			& \shortstack{0.9992\\{\tiny (0.0000)}} \\
			\bottomrule
		\end{tabular}
	\end{threeparttable}
\end{table}

\begin{table}[htbp]
	\centering
	\caption{Computational time at different measurement ratios (noise level $\sigma=10^{-2}$)}
	\label{tab:time_results}
	\begin{threeparttable}
		\scriptsize
		\renewcommand{\arraystretch}{1.4}
		\begin{tabular}{@{}l *{10}{c}@{}}
			\toprule
			\multirow{2}{*}{Method} & \multicolumn{10}{c}{Measurement Ratio $(n/d)$} \\
			\cmidrule(l){2-11}
			& {0.1} & {0.2} & {0.3} & {0.4} & {0.5} & {0.6} & {0.7} & {0.8} & {0.9} & {1.0} \\
			\midrule
			$\ell_1$    & \shortstack{0.5246\\{\tiny (0.0184)}} 
			& \shortstack{0.9662\\{\tiny (0.0667)}} 
			& \shortstack{1.4456\\{\tiny (0.0270)}} 
			& \shortstack{2.2435\\{\tiny (0.4514)}} 
			& \shortstack{3.4744\\{\tiny (0.1729)}} 
			& \shortstack{4.0473\\{\tiny (0.1358)}} 
			& \shortstack{4.6626\\{\tiny (0.1794)}} 
			& \shortstack{4.5466\\{\tiny (1.5119)}} 
			& \shortstack{4.5514\\{\tiny (0.4814)}} 
			& \shortstack{4.3822\\{\tiny (0.5646)}} \\
			$\ell_{1/2}$ & \shortstack{0.6842\\{\tiny (0.0559)}} 
			& \shortstack{1.1321\\{\tiny (0.0987)}} 
			& \shortstack{1.5254\\{\tiny (0.1543)}} 
			& \shortstack{1.3441\\{\tiny (0.5833)}} 
			& \shortstack{1.3820\\{\tiny (0.9612)}} 
			& \shortstack{0.9934\\{\tiny (0.4020)}} 
			& \shortstack{0.8612\\{\tiny (0.1362)}} 
			& \shortstack{0.7285\\{\tiny (0.2167)}} 
			& \shortstack{0.7256\\{\tiny (0.1145)}} 
			& \shortstack{0.7511\\{\tiny (0.0697)}} \\
			SCAD  & \shortstack{0.2307\\{\tiny (0.0683)}} 
			& \shortstack{0.6189\\{\tiny (0.1964)}} 
			& \shortstack{1.1997\\{\tiny (0.3138)}} 
			& \shortstack{2.3479\\{\tiny (0.5072)}} 
			& \shortstack{3.4289\\{\tiny (0.3615)}} 
			& \shortstack{3.8048\\{\tiny (0.4553)}} 
			& \shortstack{3.7234\\{\tiny (0.6130)}} 
			& \shortstack{2.9868\\{\tiny (0.7921)}} 
			& \shortstack{2.9846\\{\tiny (0.6118)}} 
			& \shortstack{2.9563\\{\tiny (0.7427)}} \\
			MCP   & \shortstack{0.1622\\{\tiny (0.0460)}} 
			& \shortstack{0.3774\\{\tiny (0.1357)}} 
			& \shortstack{0.8955\\{\tiny (0.3318)}} 
			& \shortstack{2.1388\\{\tiny (0.5367)}} 
			& \shortstack{2.6129\\{\tiny (0.5691)}} 
			& \shortstack{2.7603\\{\tiny (0.7117)}} 
			& \shortstack{2.7807\\{\tiny (0.8023)}} 
			& \shortstack{2.1441\\{\tiny (0.6194)}} 
			& \shortstack{2.1697\\{\tiny (0.4512)}} 
			& \shortstack{2.2514\\{\tiny (0.5815)}} \\
			TL1   & \shortstack{2.7968\\{\tiny (0.5581)}} 
			& \shortstack{2.1315\\{\tiny (0.6213)}} 
			& \shortstack{4.2715\\{\tiny (1.8003)}} 
			& \shortstack{7.3948\\{\tiny (0.3143)}} 
			& \shortstack{7.3901\\{\tiny (1.0583)}} 
			& \shortstack{7.2175\\{\tiny (1.5624)}} 
			& \shortstack{6.0507\\{\tiny (1.6447)}} 
			& \shortstack{7.5181\\{\tiny (6.8340)}} 
			& \shortstack{5.4924\\{\tiny (0.5665)}} 
			& \shortstack{5.1223\\{\tiny (0.8201)}} \\
			LOG   & \shortstack{0.5163\\{\tiny (0.0153)}} 
			& \shortstack{0.9778\\{\tiny (0.0968)}} 
			& \shortstack{1.3788\\{\tiny (0.1775)}} 
			& \shortstack{1.7594\\{\tiny (0.7271)}} 
			& \shortstack{1.1436\\{\tiny (0.6788)}} 
			& \shortstack{1.0426\\{\tiny (0.7505)}} 
			& \shortstack{0.7997\\{\tiny (0.1000)}} 
			& \shortstack{0.7236\\{\tiny (0.2294)}} 
			& \shortstack{0.6966\\{\tiny (0.0640)}} 
			& \shortstack{0.7166\\{\tiny (0.0405)}} \\
			EXP   & \shortstack{0.5661\\{\tiny (0.0163)}} 
			& \shortstack{0.9938\\{\tiny (0.0654)}} 
			& \shortstack{1.4315\\{\tiny (0.0442)}} 
			& \shortstack{2.1014\\{\tiny (0.4496)}} 
			& \shortstack{2.4284\\{\tiny (0.4303)}} 
			& \shortstack{2.4131\\{\tiny (0.4705)}} 
			& \shortstack{2.4388\\{\tiny (0.8907)}} 
			& \shortstack{1.9862\\{\tiny (0.7783)}} 
			& \shortstack{1.9755\\{\tiny (0.2541)}} 
			& \shortstack{1.9815\\{\tiny (0.2683)}} \\
			\bottomrule
		\end{tabular}
	\end{threeparttable}
\end{table}

\subsection{Discussion.}

To summarize, the above experiments consistently demonstrate the advantages of WCCP for sparse quadratic measurements model. Compared with classical $\ell_1$ and $\ell_{1/2}$ regularization, WCCP methods achieve more accurate support recovery, lower estimation error, and substantially reduced shrinkage bias, while maintaining competitive computational efficiency. These improvements are particularly evident in the moderately undersampled regime, where convex regularization tends to suffer from large bias and excessive false positives.

The empirical results also reveal a clear relationship between the concavity of the penalty functions and their recovery performance. As characterized by the standardized concavity parameter $\kappa$, penalties with stronger concavity generally provide a closer approximation to the $\ell_0$ regularizer and therefore exhibit improved support identification and reduced estimation bias. In particular, TL1 and LOG, which possess relatively stronger concavity among the considered penalties, consistently achieve the highest success rates and the most accurate support recovery across a wide range of sampling ratios. By contrast, the convex $\ell_1$ penalty corresponds to zero concavity and tends to introduce significant shrinkage bias, often leading to a large number of false positives. Penalties with moderate concavity, such as SCAD and MCP, offer a favorable compromise between statistical accuracy and computational efficiency, producing stable recovery performance while remaining relatively inexpensive computationally.

From a practical perspective, different penalties provide distinct trade-offs between recovery accuracy and computational cost. TL1 and LOG tend to perform best in strongly undersampled regimes where accurate support recovery is critical. SCAD and MCP, on the other hand, achieve competitive reconstruction quality with significantly lower computational cost and may therefore be preferable in large-scale applications. The $\ell_{1/2}$ penalty is computationally efficient but appears less stable when the number of measurements is small.

Overall, these empirical findings are consistent with the theoretical properties established earlier in the paper. In particular, the results support the theoretical insight that WCCP provides a closer approximation to the $\ell_0$ penalty while preserving tractable optimization. This combination of favorable statistical properties and practical computational behavior makes WCCP a promising framework for sparse quadratic measurements model.

	\appendix
	\section{Proofs of Theorem 1 and 2}
Without loss of generality, let $\Gamma^\star=\{1,...,s\}$ and $\bm\beta^\star=(\bm\beta_1^{\star \top },\bm 0^\top )^\top $, and then denote
\begin{equation*}
	\bm Z_i=
	\begin{bmatrix}
		\bm Z_i^{11}&\bm Z_i^{12}\\
		\bm Z_i^{21}&\bm Z_i^{22}
	\end{bmatrix}.
\end{equation*}
For convenience, we denote
\begin{equation*}
	F(\bm\beta ):=L(\bm\beta)+P_{\lambda_n}(\bm\beta),
\end{equation*} where $L_n(\bm\beta):=\frac{1}{4n}\sum\limits_{i=1}^n({{\bm\beta}^\top \bm{Z}_i{\bm\beta}-y_i})^2$. 
Since $p_\lambda'(t)$ is locally Lipschitz continuous on $(0,\infty)$, by Rademacher's theorem,
the second derivative $p_\lambda''(t)$ exists almost everywhere. For any $t_0>0$,
the Clarke generalized second-order derivative of
$p_\lambda$ at $t_0$ is defined as
\[
\partial^2 p_\lambda(t_0)
=
\operatorname{co}\Big\{
\lim_{k\to\infty} p_\lambda''(t_k):
t_k\to t_0,\;
p_\lambda''(t_k)\ \text{exists}
\Big\},
\]
where $\operatorname{co}$ denotes the convex hull (all convex combinations of the set).
Moreover, since $p_\lambda$ is $\mu$-weakly convex,
its Clarke generalized second-order derivative satisfies
\[
\partial^2 p_\lambda(t)\subset[-\mu,+\infty), \quad \forall\, t>0.
\]


The following lemma is analogous to Lemma 3.1 in \cite{Fan2025} but uses a different tail bound for the noise term, which leads to a slightly different probability lower bound and a different constant in the deviation inequality.
\begin{lemma}\label{lemma 1} 		Under Assumptions 2-4, if $n\geq 3$ and $c_4\geq3/(ns)$ it follows that
	\begin{equation*}
		\mathbb{P}(E^1)\geq1-2\exp \{-d \ln(1+2n)\}-\frac{1}{n^{-n/2}},
	\end{equation*}where the event is defined by
	\begin{equation*}
		E^1:=\{\sup_{\|\bm u\|=1,\|\bm v\|=1}|\frac{1}{n}\sum_{i=1}^{n}\bm u^\top \bm Z_i^{11}\bm v\varepsilon_i|\leq 5 \sigma \sqrt{c_4}\sqrt{\frac{s\ln (1+2n)}{n}}\}.
	\end{equation*}  
\end{lemma}
\begin{proof}

The proof follows the same structure as Lemma 3.1 in \cite{Fan2025}. The only difference is our choice of the parameter $t$ when bounding $\frac{1}{n}\sum_{i=1}^n\varepsilon_i^2$ via Bernstein's inequality. While \cite{Fan2025} uses $t = \sigma^2$, we take $t = 16\sigma^2\ln n$. This yields a different tail probability $n^{-n/2}$ and, under the condition $c_4 \geq 3/(ns)$, leads to the deviation inequality
\[
\sup_{\|\bm u\|=1,\|\bm v\|=1} \Bigl|\frac{1}{n}\sum_{i=1}^{n}\bm u^\top \bm Z_i^{11}\bm v\varepsilon_i\Bigr|
\leq 5\sigma\sqrt{c_4}\sqrt{\frac{s\ln(1+2n)}{n}}.
\]

We now detail the modified steps. By Assumption 2 and Lemma 1.12 of \cite{RH15}, $\varepsilon_i^2 - \mathbb{E}\varepsilon_i^2$ is sub‑exponential with parameter $16\sigma^2$. Applying Bernstein's inequality (Theorem 1.13 in \cite{RH15}) with $t = 16\sigma^2\ln n$, we obtain
\[
\mathbb{P}\Bigl(\frac{1}{n}\sum_{i=1}^n\bigl(\varepsilon_i^2 - \mathbb{E}\varepsilon_i^2\bigr) > 16\sigma^2\ln n\Bigr)
\leq \exp\bigl\{-n\ln n/2\bigr\} = n^{-n/2}.
\]
Since $\sigma^2 \geq \mathbb{E}\varepsilon_1^2$ (due to the sub‑Gaussian property), it follows that
\begin{equation}\label{ei-bound}
	\mathbb{P}\Bigl(\frac{1}{n}\sum_{i=1}^n\varepsilon_i^2 > \sigma^2 + 16\sigma^2\ln n\Bigr) \leq n^{-n/2}.
\end{equation}

The remaining arguments are identical to those in \cite{Fan2025}. In particular, with probability at least $1 - 2\exp\{-d \ln(1+2n)\} - n^{-n/2}$,
\[
\sup_{\|\bm u\|=1,\|\bm v\|=1}
\Bigl|\frac{1}{n}\sum_{i=1}^{n}\bm u^\top \bm Z_i^{11}\bm v\varepsilon_i\Bigr|
\leq \sigma\sqrt{6c_4}\sqrt{\frac{s\ln(1+2n)}{n}} + \frac{\sqrt{\sigma^2 + 16\sigma^2\ln n}}{n}.
\]
Using $c_4 \geq 3/(ns)$, we have
\[
\frac{\dfrac{\sqrt{\sigma^2 + 16\sigma^2\ln n}}{n}}
{\sigma\sqrt{6c_4}\sqrt{\dfrac{s\ln(1+2n)}{n}}}
\leq \frac{\sqrt{3\ln n}}{\sqrt{c_4 n s \ln(1+2n)}} \leq 1,
\]
which implies $\frac{\sqrt{\sigma^2 + 16\sigma^2\ln n}}{n} \leq \sigma\sqrt{6c_4}\sqrt{\frac{s\ln(1+2n)}{n}}$. Consequently,
\[
\sup_{\|\bm u\|=1,\|\bm v\|=1}
\Bigl|\frac{1}{n}\sum_{i=1}^{n}\bm u^\top \bm Z_i^{11}\bm v\varepsilon_i\Bigr|
\leq 2\,\sigma\sqrt{6c_4}\sqrt{\frac{s\ln(1+2n)}{n}}
\leq 5\sigma\sqrt{c_4}\sqrt{\frac{s\ln(1+2n)}{n}}.
\]
This completes the proof.
\end{proof}

\noindent{\bf Proof of Theorem 1} Without loss of generality, we assume $\|\bm {\hat\beta}_{1}^o-\bm \beta_1^\star\|\leq \|\bm {\hat\beta}_{1}^o+\bm \beta_1^\star\|$. Subsequently, it is easy to verify $\|\bm \beta_1^\star \|\leq \|\bm {\hat\beta}_{1}^o+\bm \beta_1^\star\|$.

Consider the level set $M=\{\bm \beta_1\in \mathbb{R}^s:\tilde{L}_n(\bm \beta_1)+\tilde{P}_{\lambda_n}(\bm \beta_1)\leq \tilde{L}_n(\bm \beta_1^\star)+\tilde{P}_{\lambda_n}(\bm \beta_1^\star)\}$. It is clear that
$$\inf_{\bm \beta_1\in \mathbb{R}^s} \tilde{L}_n(\bm \beta_1)+\tilde{P}_{\lambda_n}(\bm \beta_1)=\inf_{\bm \beta_1\in M} \tilde{L}_n(\bm \beta_1)+\tilde{P}_{\lambda_n}(\bm \beta_1).$$
Since $\tilde{L}_n(\cdot)+\tilde{P}_{\lambda_n}(\cdot)$ is continuous and the level set is compact, there exists at least one minimizer $\bm {\hat\beta}_1$ in the level set.

Noting that condition (iv) of Assumption 1 implies that all sub-differential and derivatives of $P_{\lambda_n}$ are bounded in magnitude by $\lambda_n \varrho$, we get
\begin{equation}\label{S1.1}
	|\tilde{P}_{\lambda_n}(\bm {\hat\beta}_{1}^o)-\tilde{P}_{\lambda_n}(\bm \beta_1^\star )|
	\leq \lambda_n \varrho\|\bm {\hat\beta}_{1}^o-\bm \beta_1^\star \|_1
	\leq \lambda_n \sqrt{s} \varrho\| \bm {\hat\beta}_{1}^o-\bm \beta_1^\star\|.
\end{equation}

According to the definition of $\bm {\hat\beta}_{1}^o$, for any $\bm \beta_1\in \mathbb{R}^s$,
\begin{equation*}
	\tilde{L}_n(\bm {\hat\beta}_{1}^o)+\tilde{P}_{\lambda_n}(\bm {\hat\beta}_{1}^o)\leq  \tilde{L}_n(\bm \beta_1)+\tilde{P}_{\lambda_n}(\bm \beta_1).
\end{equation*}
Then we get 
\begin{align*}
	0 &\geq \tilde{L}_n(\bm {\hat\beta}_{1}^o)+\tilde{P}_{\lambda_n}(\bm {\hat\beta}_{1}^o)-\tilde{L}_n(\bm \beta_1^\star)-\tilde{P}_{\lambda_n}(\bm \beta_1^\star) \\
	&= \frac{1}{4n}\sum_{i=1}^{n}(\bm {\hat\beta}_{1}^{o\top}\bm Z_i^{11}\bm {\hat\beta}_{1}^o-y_i)^2
	-\frac{1}{4n}\sum_{i=1}^{n}( \bm \beta_1^{\star\top}\bm Z_i^{11}\bm \beta_1^\star-y_i)^2
	+\tilde{P}_{\lambda_n}(\bm {\hat\beta}_{1}^o)-\tilde{P}_{\lambda_n}(\bm \beta_1^\star) \\
	&= \frac{1}{4n}\sum_{i=1}^{n}\bigl[(\bm{\hat\beta}_{1}^o+\bm \beta_1^\star)^\top \bm Z_i^{11}(\bm{\hat\beta}_{1}^o-\bm \beta_1^\star)\bigr]^2 \\
	&\quad -\frac{1}{2n}\sum_{i=1}^{n}(\bm{\hat\beta}_{1}^o+\bm\beta_1^\star)^\top \bm Z_i^{11}(\bm{\hat\beta}_{1}^o-\bm \beta_1^\star)\varepsilon_i \\
	&\quad +\tilde{P}_{\lambda_n}(\bm{\hat\beta}_{1}^o)-\tilde{P}_{\lambda_n}(\bm\beta_1^\star) \\
	&\geq \frac{c_3}{4}\|\bm{\hat\beta}_{1}^o+\bm \beta_1^\star\|^2\|\bm {\hat\beta}_{1}^o-\bm \beta_1^\star\|^2 \\
	&\quad -\frac{1}{2n}\sum_{i=1}^{n}(\bm{\hat\beta}_{1}^o+\bm\beta_1^\star)^\top \bm Z_i^{11}(\bm{\hat\beta}_{1}^o-\bm\beta_1^\star)\varepsilon_i \\
	&\quad +\tilde{P}_{\lambda_n}(\bm{\hat\beta}_{1}^o)-\tilde{P}_{\lambda_n}(\bm\beta_1^\star).
\end{align*}
Combining the inequality \eqref{S1.1}, we obtain
\begin{equation*}
\begin{aligned}
&\frac{c_3}{4}\|\bm{\hat\beta}_{1}^o-\bm\beta_1^\star\|^2\|\bm{\hat\beta}_{1}^o+\bm\beta_1^\star\|^2\\
\leq&\frac{1}{2n}\|\bm{\hat\beta}_{1}^o-\bm\beta_1^\star\|\|\bm{\hat\beta}_{1}^o+\bm\beta_1^\star\|\sum_{i=1}^{n}(\frac{\bm{\hat\beta}_{1}^o-\bm\beta_1^\star}{\|\bm {\hat\beta}_{1}^o-\bm\beta_1^\star\|})^\top \bm Z_i^{11}(\frac{\bm{\hat\beta}_{1}^o+\bm\beta_1^\star}{\|\bm{\hat\beta}_{1}^o+\bm\beta_1^\star\|})\varepsilon_i\\
&\qquad+\lambda_n\sqrt{s}\varrho\|\bm {\hat\beta}_{1}^o-\bm\beta_1^\star \|.
\end{aligned}
\end{equation*}

Under the event $E^1$, we get
\begin{equation*}
\|\bm{\hat\beta}_{1}^o-\bm\beta_1^\star\|\leq \frac{10\sigma\sqrt{c_4}}{c_3\|\bm\beta^\star\|}\sqrt{\frac{s\ln(1+2n)}{n}}+\frac{4\lambda_n \varrho\sqrt{s}}{c_3\|\bm\beta ^\star\|^2}.
\end{equation*} 
which yields that
\begin{eqnarray}\label{oracle-bound}
\|\bm{\hat\beta}_{1}^o-\bm\beta_1^\star\|\leq& \frac{10\sigma\sqrt{c_4}}{c_1c_3}\sqrt{\frac{\ln(1+2n)}{n}}+\frac{4\lambda_n\varrho}{c_1^2c_3\sqrt{s}}\nonumber\\
<&C_0\big(\sqrt{\frac{\ln(1+2n)}{n}}+\frac{\lambda_n \varrho}{\sqrt{s}}\big),
\end{eqnarray} 
where $C_0=\max\left\{10\sigma\sqrt{c_4}/c_1c_3,4/c_1^2c_3\right\}$. Combing this and  Lemma \ref{lemma 1}, we get the desired result.
\qed

To prove Theorem 2, we also need the following lemma.	
\begin{lemma}\label{prob-E2}
Define 
$$
E^2=:\Big\{\max_{1\leq j\leq d-s} \sup_{\|\bm u\|\leq r_n}|\frac{1}{n}\sum_{i=1}^{n}\bm e_{d-s,j}^\top \bm Z_i^{21}\bm u\varepsilon_i|\leq t_2+\frac{\sqrt{s(\sigma^2+16\sigma^2\ln n)}}{n}\Big\}
$$with $t_2=\sqrt{2}\sigma r_n\sqrt{\frac{s(s+1)\ln(1+2n)+s\ln d}{n}}$.	Under Assumptions 2-4, it follows that 	
$$\mathbb{P}\big(E^2\big)\geq1-\frac{1}{n}-\frac{1}{n^{-n/2}}.$$
\end{lemma}	
\begin{proof}
Let $\mathbb{B}(\bm u,\delta):=\left\{\bm v\in \mathbb{R}^s:\|\bm v-\bm u\|\leq \delta\right\}$. Then follows from the Lemma 14.27 of \citet*{b2011} that there exist $K:=(1+2n)^s$ vectors $\bm{u_k}$ with $\bm{u_k}\in \mathbb{R}^s$ and $\|\bm{u_k}\|\leq r_n$ such that 
\begin{equation*}
\left\{\bm u\in\mathbb{R}^s:~~~\|\bm u\|\leq r_n\right\}\subseteq \cup_{k=1}^{K}\mathbb{B}(\bm{u_k},\frac{1}{n}).
\end{equation*} 
Using this fact, we obtain the following inequalities
\begin{equation}\label{Zu-rn-bound}
\begin{split}
	&	\max_{1\leq j\leq d-s}	\sup_{\|\bm u\|\leq r_n}|\frac{1}{n}\sum_{i=1}^{n}\bm e_{d-s,j}^\top \bm Z_i^{21}\bm u\varepsilon_i|
	\\	\leq & \max_{1\leq j\leq d-s} \sup_{\bm u\in \cup_{k=1}^{K}\mathbb{B}(\bm u_k,\frac{1}{n}) }|\frac{1}{n}\sum_{i=1}^{n}\bm e_{d-s,j}^\top \bm Z_i^{21}\bm u\varepsilon_i|
	\\ \leq & \max_{1\leq j\leq d-s}\max_{1\leq k\leq K}|\frac{1}{n}\sum_{i=1}^{n}\bm e_{d-s,j}^\top \bm Z_i^{21}\bm{u_k}\varepsilon_i|\\
	&\qquad+\max_{1\leq j\leq d-s}\sup_{\bm u\in\mathbb{B}(\bm{u_k},\frac{1}{n}) }|\frac{1}{n}\sum_{i=1}^{n}\bm e_{d-s,j}^\top \bm{Z}_i^{21}(\bm u-\bm{u_k})\varepsilon_i|.
\end{split}
\end{equation}
Noting the fact that $|\bm{\beta}^\top \bm{Z}\bm{\beta}'|\leq|\bm{Z}|_\infty\|\bm{\beta}\|_1\|\bm{\beta}'\|_1$ for any $n\times{}d$ matrix $\bm{Z}$ and vectors $\bm{\beta},\,\bm{\beta}'\in\mathbb{R}^d$, one can conclude from (\ref{Z-std}) that 
\begin{equation}\label{z21_ub}\sum_{i=1}^{n}(\bm e_{d-s,j}^\top \bm Z_i^{21}\bm{u})^2\leq\sum_{i=1}^{n}\|\bm e_{d-s,j}\|_1^2|\bm Z_i^{21}|_\infty^2\|\bm{u}\|_1^2
\leq\sum_{i=1}^{n}|\bm Z_i|_\infty^2\|\bm{u}\|_1^2\leq n\|\bm{u}\|_1^2\end{equation}	
for any $\bm{u}\in\mathbb{R}^s$. 

For the first term of (\ref{Zu-rn-bound}), we can apply  the countable subadditivity of probability, Bernstein's inequality, inequality (\ref{z21_ub}) and 
$$\|\bm{u}_k\|_1\leq\sqrt{s}\|\bm{u}_k\|\leq \sqrt{s}r_n$$ to obtain, for any $t>0$,
\begin{equation*}
\begin{split}
	&	\mathbb{P}\Big(\max_{1\leq j\leq d-s} \max_{1\leq k\leq K}|\frac{1}{n}\sum_{i=1}^{n}\bm e_{d-s,j}^\top \bm Z_i^{21}\bm{u_k}\varepsilon_i|>t\Big)
	\\	\leq &   \sum_{j=1}^{d-s}\sum_{k=1}^{K}\mathbb{P}\big( |\frac{1}{n}\sum_{i=1}^{n}\bm e_{d-s,j}^\top \bm Z_i^{21}\bm{u_k}\varepsilon_i|>t\big)
	\\ \leq &2dK \exp(-\frac{n^2t^2}{2\sigma^2\sum_{i=1}^{n}(\bm e_{d-s,j}^\top \bm Z_i^{21}\bm{u_k})^2})
	\\ \leq &2dK \exp(-\frac{nt^2}{2\sigma^2sr_n^2})
	\\	 \leq  & 2\exp \left\{-\frac{nt^2}{2\sigma^2sr_n^2}+s\ln(1+2n)+\ln d\right\}.
\end{split}
\end{equation*}
Taking $t=t_2$, we drive \begin{equation*}
\begin{split}
	&	\mathbb{P}\Big(\max_{1\leq j\leq d-s} \max_{1\leq k\leq K}|\frac{1}{n}\sum_{i=1}^{n}\bm e_{d-s,j}^\top \bm Z_i^{21}\bm{u_k}\varepsilon_i|>t_2\Big)<\frac{1}{n}.
\end{split}
\end{equation*}

For the second term of (\ref{Zu-rn-bound}), we can calculate that 
\begin{equation}\label{uun-bound}
\begin{split}
	&\max_{1\leq j\leq d-s} \sup_{\bm u\in \mathbb{B}(\bm{u_k},\frac{1}{n})}|\frac{1}{n}\sum_{i=1}^{n}\bm e_{d-s,j}^\top \bm Z_i^{21}(\bm u-\bm{u_k})\varepsilon_i|
	\\ \leq & \max_{1\leq j\leq d-s} \sup_{\bm u\in \mathbb{B}(\bm{u_k},\frac{1}{n})}\sqrt{\frac{1}{n}\sum_{i=1}^{n}[\bm e_{d-s,j}^\top \bm Z_i^{21}(\bm u-\bm{u_k})]^2}\sqrt{\frac{1}{n}\sum_{i=1}^{n}\varepsilon_i^2}
	\\ \leq & \max_{1\leq j\leq d-s} \sup_{\bm u\in \mathbb{B}(\bm{u_k},\frac{1}{n})}\sqrt{\frac{1}{n}\sum_{i=1}^{n}\|\bm e_{d-s,j}\|_1^2|\bm Z_i^{21}|_\infty^2\|\bm u-\bm{u_k}\|_1^2 }\sqrt{\frac{1}{n}\sum_{i=1}^{n}\varepsilon_i^2} \\
	\leq & \sup_{\bm u\in \mathbb{B}(\bm{u_k},\frac{1}{n})}\sqrt{\frac{s}{n}\sum_{i=1}^{n}|\bm Z_i^{21}|_\infty^2\|\bm u-\bm{u_k}\|_2^2 }\sqrt{\frac{1}{n}\sum_{i=1}^{n}\varepsilon_i^2}
	\\ \leq &\frac{\sqrt{s}}{n}\sqrt{\frac{1}{n}\sum_{i=1}^{n}\varepsilon_i^2}.
\end{split}
\end{equation}

Combing this, inequalities (\ref{ei-bound}), (\ref{Zu-rn-bound}) and (\ref{uun-bound}), get the desired result.
\end{proof}

%

\begin{lemma}\label{prob-E3} 
Under Assumptions 2-4, it follows that 
\begin{equation*}
\mathbb{P}\big(E^3\big)\geq1-\frac{1}{n}
\end{equation*}
where the event $E^3$ is defined by
\begin{equation*}
E^3:=\{\max_{1\leq j\leq d-s}|\frac{1}{n}\sum_{i=1}^{n}\bm e_{d-s,j}^\top \bm Z_i^{21}\bm \beta_1^\star\varepsilon_i|\leq\sqrt{2} c_2\sigma s\sqrt{\frac{\ln (1+2n)+\ln d}{n}}\}. 
\end{equation*}		
\end{lemma}

\begin{proof}
From Bernstein's inequality, (\ref{z21_ub}) and (\ref{(3)}), we can conclude  that
\begin{equation*}
	\begin{split}
		\mathbb{P}\Big(\max_{1\leq j\leq d-s}|\frac{1}{n}\sum_{i=1}^{n}\bm e_{d-s,j}^\top \bm Z_i^{21}\bm \beta_1^\star\varepsilon_i|>t\Big)
		\leq & 2d \exp \Big\{\frac{-n^2t^2}{2\sigma^2\sum_{i=1}^{n}(\bm e_{d-s,j}^\top \bm Z_i^{21}\bm \beta_1^\star)^2}\Big\}
		\\ \leq & 2d \exp \Big\{\frac{-nt^2}{2\sigma^2\|\bm \beta_1^\star\|_1^2}\Big\}
		\\ \leq & 2 \exp \big\{\frac{-nt^2}{2\sigma^2c_2^2s^2}+\ln d\big\}.
	\end{split}
\end{equation*}
By taking $t=t_1:=\sqrt{2} c_2\sigma s\sqrt{\frac{\ln (1+2n)+\ln d}{n}}$, we subsequently obtain
\begin{equation*}
	\mathbb{P}\Big(\max_{1\leq j\leq d-s}|\frac{1}{n}\sum_{i=1}^{n}\bm e_{d-s,j}^\top \bm Z_i^{21}\bm \beta_1^\star\varepsilon_i|>t_1\Big)<\frac{1}{n}.
\end{equation*}
Then, we get the desired result.
\end{proof}

By combining the ideas from the proofs of Theorem 3.3 in \cite{chen2013} and Theorem 2.4 in \cite{bai2024avoiding}, we derive the following result.
\begin{lemma}\label{secondcon-opt}
Suppose Assumption 1 holds. Let \(\bm{\beta}\) be a stationary point of Problem~\eqref{prob}, i.e.,
\begin{equation}
	0 \in \nabla L_n(\bm{\beta}) + \partial P_{\lambda_n}(\bm{\beta}).\nonumber
\end{equation}
If for any nonzero vector \(\bm{g} \in \mathbb{R}^d\) with \(g_j = 0\) for all \(j \notin \Gamma^{*}\),
\begin{equation}\label{second-order-con}
	\bm{g}^\top \nabla^2 L_n(\bm{\beta}) \bm{g} + \sum_{j \in \Gamma^*}  \min \partial^2 p_{\lambda_n}(|\beta_j|)\, g_j^2 > 0,
\end{equation}
then \(\bm{\beta}\) is a strict local minimum of Problem~\eqref{prob}.
\end{lemma}

\noindent{\bf Proof of Theorem 2} 	According to Theorem 1, we can obtain the existence of a estimator $\bm{\hat\beta} _1^o \in \mathbb{R}^s$, and the error bound between $\bm{\hat\beta} _1^o$ and $\bm\beta_1^\star $ is $r_n$. 
Let $\bm{\hat\beta}$ be a vector with $\bm{\hat\beta} _{\Gamma^{\star}}=\bm{\hat\beta}_1^o$ and $\bm{\hat\beta} _{\Gamma^{\star c}}=0$.

Firstly, we prove that $\bm {\hat\beta}$ is a stationary point of Problem~\eqref{prob}, i.e.,
\begin{equation}\label{sta-WCCR}
0\in \nabla L_n(\bm{\hat{\beta}})+\partial P_{\lambda_n}(\bm{\hat{\beta}}).
\end{equation}	
By simple caiculation, we have
\begin{equation*}
\begin{split}
	\nabla L_n(\bm{\hat\beta})+\partial P_{\lambda_n}(\bm {\hat\beta})
	&= \begin{bmatrix}
		\frac{1}{n}\sum_{i=1}^{n}(\bm{\hat\beta}_1^{o \top }\bm Z_i^{11}\bm {\hat\beta}_1^o-y_i)\bm Z_i^{11}\bm{\hat\beta}_1^o \\
		\frac{1}{n}\sum_{i=1}^{n}(\bm{\hat\beta}_1^{o\top }\bm Z_i^{11}\bm{\hat\beta}_1^o-y_i)\bm Z_i^{21}\bm{\hat\beta}_1^o
	\end{bmatrix}+
	\begin{bmatrix}
		\sum_{j=1}^{s}p_{\lambda_n}^{'}(\hat\beta_j)\\
		\sum_{j=s+1}^{d}\partial p_{\lambda_n}(\hat\beta_j)
	\end{bmatrix}
	\\&=\begin{bmatrix}
		\nabla \tilde{L}_n(\bm{\hat\beta}_1^o)+\nabla \tilde{P}_{\lambda_n}(\bm{\hat\beta}_1^o)\\
		\nabla L_n(\bm{\hat{\beta}})_{\Gamma^{\star c}}+\partial P_{\lambda_n}(\bm{\hat{\beta}})_{\Gamma^{\star c}}
	\end{bmatrix}.
\end{split}				
\end{equation*}
Since $\nabla \tilde{L}_n(\bm{\hat\beta}_1^o)+\nabla \tilde{P}_{\lambda_n}(\bm{\hat\beta}_1^o)=0,$  we only need to prove
\begin{equation*}
0\in \left[\nabla L_n(\bm{\hat{\beta}})+\partial P_{\lambda_n}(\bm{\hat{\beta}})\right]_{\Gamma^{\star c}}.
\end{equation*}
It suffices to prove that
\begin{equation}\label{first-order-cond}
\|\nabla L_n(\bm{\hat{\beta}})_{\Gamma^{\star c}}\|_\infty\leq \lambda_n \varrho,
\end{equation}
since $\partial p_{\lambda_n}(|\hat\beta_j|)\in [-\lambda_n \varrho,\lambda_n \varrho]$. 
Noting	\begin{equation*}
\begin{split}
	\nabla L_n(\bm{\hat{\beta}})_{\Gamma^{\star c}}&=\frac{1}{n}\sum_{i=1}^{n}(\bm{\hat{\beta}}^\top\bm Z_i\bm{\hat{\beta}}-y_i)\bm Z_i^{21}\bm{\hat\beta}_1^o \\& =\frac{1}{n}\sum_{i=1}^{n}(\bm{\hat\beta}_1^o-\bm\beta_1^\star)^\top \bm Z_i^{11}(\bm{\hat\beta}_1^o+\bm \beta_1^\star)Z_i^{21}\bm {\hat\beta}_1^o-\frac{1}{n}\sum_{i=1}^{n}Z_i^{21}\bm{\hat\beta} _1^o\varepsilon_i,
\end{split}
\end{equation*}
we get
\begin{eqnarray}\label{grad-bound}
\begin{aligned}
	&\|\nabla L_n(\bm{\hat{\beta}})_{\Gamma^{\star c}}\|_\infty\\ \leq& \|\frac{1}{n}\sum_{i=1}^{n}(\bm{\hat\beta}_1^o-\bm \beta_1^\star)^\top \bm Z_i^{11}(\bm{\hat\beta}_1^o+\bm \beta_1^\star)\bm Z_i^{21}\bm{\hat\beta}_1^o\|_\infty+\|\frac{1}{n}\sum_{i=1}^{n}\bm Z_i^{21}\bm{ \hat\beta}_1^o\varepsilon_i\|_\infty.
\end{aligned}
\end{eqnarray}
We first estimate the upper bound of the first term on the right-hand side of the above inequality. Using again the fact that $|\bm{\beta}^\top \bm{Z}\bm{\beta}'|\leq|\bm{Z}|_\infty\|\bm{\beta}\|_1\|\bm{\beta}'\|_1$ for any $n\times{}d$ matrix $\bm{Z}$ and vectors $\bm{\beta},\,\bm{\beta}'\in\mathbb{R}^d$,  we conclude from  H\"older's  inequality, Assumption 4  and (\ref{(3)}) that 
\begin{eqnarray}\label{grad1-bound}
\begin{aligned}
	&\|\frac{1}{n}\sum_{i=1}^{n}(\bm{\hat\beta}_1^o-\bm \beta_1^\star)^\top \bm Z_i^{11}(\bm{\hat\beta}_1^o+\bm \beta_1^\star)\bm Z_i^{21}\bm{\hat{\beta}}_1^o\|_\infty
	\\ 
	=&\max_{1\leq k\leq d-s}\big|(\bm{\hat\beta}_1^o-\bm \beta_1^\star)^\top \big(\frac{1}{n}\sum_{i=1}^{n}\bm Z_i^{11}(\bm{\hat\beta}_1^o+\bm \beta_1^\star)\bm e_{d-s,k}^\top\bm Z_i^{21}\big)\bm{\hat{\beta}}_1^o\big|\\
	\leq&\max_{1\leq k\leq d-s}\|\bm{\hat\beta}_1^o-\bm \beta_1^\star\|_1\|\bm{\hat{\beta}}_1^o\|_1\big|\frac{1}{n}\sum_{i=1}^{n}\bm Z_i^{11}(\bm{\hat\beta}_1^o+\bm \beta_1^\star)\bm e_{d-s,k}^\top\bm Z_i^{21}\big|_\infty
	\\ 
	= &  \|\bm{\hat\beta}_1^o-\bm \beta_1^\star\|_1\|\bm{\hat\beta}_1^o\|_1\max_{1\leq k\leq d-s}\max_{1\leq j\leq s}\max_{1\leq l\leq s}\big|\frac{1}{n}\sum_{i=1}^{n} \bm (\bm{\hat\beta}_1^o+\bm\beta_1^\star)^\top \bm{Z}_i^{11}\bm e_{s,j} \bm e_{d-s,k}^\top\bm Z_i^{21}\bm e_{s,l}\big|\\		
	\leq & \|\bm{\hat\beta}_1^o-\bm{\beta}_1^\star\|_1 |\bm{\hat\beta}_1^o\|_1\|\bm{\hat\beta}_1^o+\bm \beta_1^\star\|_1\max_{1\leq k\leq d-s}\max_{1\leq j\leq s}\big|\frac{1}{n}\sum_{i=1}^{n} \bm{Z}_i^{11}\bm e_{s,j} \bm e_{d-s,k}^\top\bm Z_i^{21}\big|_\infty\nonumber
\end{aligned}
\end{eqnarray}
\begin{eqnarray}\label{grad1-bound}
\begin{aligned}				
	\leq& s^{3/2}\|\bm{\hat\beta}_1^o+\bm \beta_1^\star\|\|\bm{\hat\beta}_1^o-\bm\beta_1^\star\| \|\bm{\hat\beta}_1^o\|  |\frac{1}{n}\sum_{i=1}^{n}\bm Z_i^{11} \otimes \bm Z_i^{21}|_\infty\\
	\leq& s^{3/2} \left(\|\bm{\hat\beta}_1^o-\bm \beta_1^\star\|+2\|\bm \beta_1^\star \|\right)\left(\|\bm{\hat\beta} _1^o-\bm\beta_1^\star\|+\|\bm \beta_1^\star \|\right)\|\bm{\hat\beta} _1^o-\bm \beta_1^\star\|\frac{c_5}{\sqrt{n}}
	\\ \leq& 6s^{3/2}\|\bm \beta_1^\star\|^2r_n\frac{c_5}{\sqrt{n}}\\
	\leq&6c_2^2s^{5/2}r_n\frac{c_5}{\sqrt{n}}
\end{aligned}
\end{eqnarray}
if the event $E^1$ occurs. Here, the fifth inequality relies on the bounds  $\|\bm{\hat{\beta}}_1^o-\bm{\beta}_1^\star\|\leq r_n$ and $r_n\leq \|\bm{\beta}_1^\star\|$. The first bound follows from (\ref{oracle-bound}) under event $E^1$. The second is a consquence of conditions (\ref{rns_con})  and $c_1\leq c_2$ which together imply
\begin{equation*}
r_n < \frac{c_3c_1^2\sqrt{s}}{2(4c_3+3c_4)c_2} < c_1\sqrt{s} \leq |\bm{\beta}_1^\star|.
\end{equation*}


We proceed to estimate the upper bound of the second term on the right-hand side of inequality (\ref{grad-bound}). Notice that 
\begin{equation*}
\begin{split}
	&\|\frac{1}{n}\sum_{i=1}^{n}\bm Z_i^{21}\bm{\hat\beta} _1^o\varepsilon_i\|_\infty\\=&\max_{1\leq j\leq d-s}|\frac{1}{n}\sum_{i=1}^{n}\bm e_{d-s,j}^\top \bm Z_i^{21}\bm{\hat\beta}_1^o\varepsilon_i|\\
	\leq &  \max_{1\leq j\leq d-s}|\frac{1}{n}\sum_{i=1}^{n}\bm e_{d-s,j}^\top \bm Z_i^{21}(\bm{\hat\beta}_1^{o}-\bm \beta_1^\star)\varepsilon_i|+\max_{1\leq j\leq d-s}|\frac{1}{n}\sum_{i=1}^{n}\bm e_{d-s,j}^\top \bm Z_i^{21}\bm \beta_1^\star\varepsilon_i|
\end{split}
\end{equation*}
\begin{equation*}
\begin{split}
	\leq & \max_{1\leq j\leq d-s}\sup_{\|\bm u\|\leq r_n}|\frac{1}{n}\sum_{i=1}^{n}\bm e_{d-s,j}^\top \bm Z_i^{21}\bm u\varepsilon_i|+\max_{1\leq j\leq d-s}|\frac{1}{n}\sum_{i=1}^{n}\bm e_{d-s,j}^\top \bm Z_i^{21}\bm \beta_1^\star\varepsilon_i|\\
	\leq&t_2+\frac{\sqrt{s(\sigma^2+t_2)}}{n}+\sqrt{2} c_2\sigma s\sqrt{\frac{\ln (1+2n)+\ln d}{n}}.
\end{split}
\end{equation*}
under the event $E^2\cap E^3$.
Therefore, combining the above inequality, inequalities (\ref{grad-bound}) and (\ref{grad1-bound}), we get 
\begin{eqnarray}\label{grad-lamb}
\begin{aligned}
	&	\|\nabla L_n(\bm{\hat{\beta}})_{\Gamma^{\star c}}\|_\infty\\ \leq& 6c_2^2s^{5/2}r_n\frac{c_5}{\sqrt{n}}+t_2+\frac{\sqrt{s(\sigma^2+16\sigma^2\ln n)}}{n}+\sqrt{2} c_2\sigma s\sqrt{\frac{\ln (1+2n)+\ln d}{n}}
\end{aligned}
\end{eqnarray}	
under the event $E^1\cap E^2 \cap E^3$.

On the other hand, note that the left inequality of condition (\ref{rns_con}) results in
\begin{equation}\label{6-inequality}
\sqrt{2}c_2\sigma s\sqrt{\frac{\ln (1+2n)+\ln d}{n}}\leq\frac{1}{6}\lambda_n\varrho. 
\end{equation}
By 		the first inequality of condition (\ref{ns_con}), we get 
$$6c_2^2c_5C_0\frac{s^{5/2}}{\sqrt{n}}\frac{\lambda_n\varrho}{\sqrt{s}}\leq\frac{1}{6}\lambda_n\varrho $$
which together with the left inequality of condition (\ref{rns_con}) also leads to
\begin{equation*}
\begin{split}6c_2^2c_5C_0\frac{s^{5/2}}{\sqrt{n}}\sqrt{\frac{\ln(1+2n)}{n}}=&\frac{6c_2^2c_5C_0s^2}{\sqrt{n}}\sqrt{\frac{s\ln(1+2n)}{n}}\\
	\leq&\frac{1}{6}\sqrt{\frac{s\ln(1+2n)}{n}}\\ \leq&\frac{1}{6}\lambda_n\varrho.
\end{split}
\end{equation*}
By simple calculation, we get 
\begin{equation*}
\begin{split}&\frac{\sqrt{2}C_0\sigma\sqrt{\frac{s(s+1)\ln(1+2n)+s\ln d}{n}}\sqrt{\frac{\ln(1+2n)}{n}}}{\sqrt{2}c_2\sigma s\sqrt{\frac{\ln(1+2n)+\ln d}{n}}}\\
	=&\frac{C_0}{c_2}\sqrt{\frac{\ln(1+2n)}{n}}\sqrt{\frac{1}{s}+\frac{\ln(1+2n)}{\ln(1+2n)+\ln d}} \\
	<&\frac{C_0}{c_2}\sqrt{\frac{\ln(1+2n)}{n}}\sqrt{2}\\ \leq&1
\end{split}
\end{equation*}
where the last ineauality follows from the second inequality of condition (\ref{ns_con}). Combing this and the inequality (\ref{6-inequality}) enable us to get
$$\sqrt{2}C_0\sigma\sqrt{\frac{s(s+1)\ln(1+2n)+s\ln d}{n}}\sqrt{\frac{\ln(1+2n)}{n}}\leq\frac{1}{6}\lambda_n\varrho.$$
By simple calculation, we get 
\begin{equation*}
\begin{split}\frac{\sqrt{2}C_0\sigma\sqrt{\frac{s(s+1)\ln(1+2n)+s\ln d}{n}}\frac{\lambda_n\varrho}{\sqrt{s}}}{\sqrt{2}c_2\sigma s\sqrt{\frac{\ln(1+2n)+\ln d}{n}}}
	=&\frac{\lambda_n\varrho C_0}{c_2\sqrt{s}}\sqrt{\frac{1}{s}+\frac{\ln(1+2n)}{\ln(1+2n)+\ln d}} \\
	<&\frac{\sqrt{2}\lambda_n\varrho C_0}{c_2\sqrt{s}}\leq1
\end{split}
\end{equation*}
where the last ineauality follows from the second inequality of condition (\ref{lambdan_con}). Combing this and the inequality (\ref{6-inequality}) enable us to get
$$\sqrt{2}C_0\sigma\sqrt{\frac{s(s+1)\ln(1+2n)+s\ln d}{n}}\times\frac{\lambda_n\varrho}{\sqrt{s}}\leq\frac{1}{6}\lambda_n\varrho.$$
Considering
\begin{equation*}
\begin{split}\frac{\frac{\sqrt{s(\sigma^2+16\sigma^2\ln n)}}{n}}{\sqrt{2}c_2\sigma s\sqrt{\frac{\ln(1+2n)+\ln d}{n}}}
	=&\frac{1}{c_2\sqrt{ns}}\times\frac{\sqrt{\frac{1}{2}+8\ln n}}{\sqrt{\ln(1+2n)+\ln d}} \\
	<&\frac{3}{c_2\sqrt{ns}}\leq1
\end{split}
\end{equation*}
where first inequality derives from $n\geq 3$(resulting in $\ln n>1/2$) and $\ln n\leq\ln(1+2n)+\ln d$ and the last ineauality follows from the second inequality of condition (\ref{lambdan_con}). Combing this and the inequality (\ref{6-inequality}) enable us to get
$$\frac{\sqrt{s(\sigma^2+16\sigma^2\ln n)}}{n}\leq\frac{1}{6}\lambda_n\varrho.$$
By these inequalities, the definition of $r_n$ and inequality (\ref{grad-lamb}), we get (\ref{first-order-cond}) under the event $E^1\cap E^2 \cap E^3$.

Next, we prove that $\bm{\hat\beta}$ is a local minimizer if the event $E^1\cap E^2 \cap E^3$ occurs. Based on Lemma \ref{secondcon-opt}, it suffices to prove that $\bm{\hat\beta}$ satisfies condition (\ref{second-order-con}).
Note that the convexity of $p_{\lambda_n}(t)+\frac{\mu}{2}t^2$ implies $$\min \partial^2 p_\lambda(t)+\mu\geq0~~ \mbox{for any}~~t\geq0,$$ and hence
\begin{equation*}
\sum_{j\in \Gamma^*}(\min \partial^2p_{\lambda_n}(|\hat{\beta}_j|)+\mu)g_j^2\geq 0.
\end{equation*} 
On the other hand, a simple calculation yields that 
\begin{equation*}
\begin{split}
	& \bm{g}^\top \nabla^2L_n(\bm{\hat{\beta}})\bm g+\sum_{j\in \Gamma^\star }\min \partial^2 p_{\lambda_n}(|\bm{\hat{\beta}}_j|)g_j^2
	\\=&\bm g^T\nabla^2L_n(\bm{\hat{\beta}})\bm g+\sum_{j\in \Gamma^\star}(\min \partial^2 p_{\lambda_n}(|\bm{\hat{\beta}}_j|)+\mu)\bm g_j^2-\sum_{j\in \Gamma^\star}\mu g_j^2.
\end{split}
\end{equation*}
Therefore, it remains to prove that for any nonzero vector \(\bm{g} \in \mathbb{R}^d\) with \(g_j = 0\) for all \(j \notin \Gamma^{*}\),
\begin{equation}\label{second-order-con2}
\bm{g}^\top \nabla^2 L_n(\hat{\bm{\beta}}) \bm{g} -\sum_{j\in \Gamma^\star}\mu g_j^2 > 0.
\end{equation}	
We now proceed to prove this inequality.	
Since $$|\frac{1}{n}\sum_{i=1}^{n}\bm g_1^\top \bm Z_i^{11}\bm g_1\varepsilon_i|\leq \|\bm g_1\|^2\sup_{\|\bm{u}\|=1,\|\bm{v}\|=1}|\frac{1}{n}\sum_{i=1}^{n}\bm{u}^\top \bm Z_i^{11}\bm{v}\varepsilon_i|,$$
we have  
$$	|\frac{1}{n}\sum_{i=1}^{n}\bm g_1^\top \bm Z_i^{11}\bm g_1\varepsilon_i|\leq 5\sigma\sqrt{c_4}\sqrt{\frac{s\ln (1+2n)}{n}}\|\bm g_1\|^2$$
if the event $E^1$ occurs.
Notice that
$$\begin{aligned}&\frac{1}{n}\sum_{i=1}^{n}(\bm{\hat\beta} _1^o+\bm\beta_1^\star)^\top \bm Z_i^{11}(\bm{\hat\beta}_1^o-\bm \beta_1^\star)g_1^\top \bm Z_i^{11}\bm g_1\\
\leq&
\frac{1}{n}\sqrt{\sum_{i=1}^{n}\big((\bm{\hat\beta} _1^o+\bm\beta_1^\star)^\top \bm Z_i^{11}(\bm{\hat\beta}_1^o-\bm \beta_1^\star)\big)^2}
\sqrt{\sum_{i=1}^{n}\big(g_1^\top \bm Z_i^{11}\bm g_1\big)^2}.
\end{aligned}$$
Combining these two inequalities and Assumption 3, under the event $E^1$, we have for any non-zero vector \(\bm{g} \in \mathbb{R}^d\) with \(g_j = 0\) for all \(j \notin \Gamma^{*}\),
\begin{equation*}
\begin{split}
	& \bm{g}^\top \nabla^2L_n(\bm{\hat{\beta}})\bm g-\sum_{j\in \Gamma^\star}\mu g_j^2
	\\=&\frac{2}{n}\sum_{i=1}^{n}\bm g^\top \bm Z_i\bm{\hat{\beta}}\bm{\hat{\beta}}^\top \bm Z_i\bm g+\frac{1}{n}\sum_{i=1}^{n}(\bm{\hat{\beta}}^\top \bm Z_i\bm{\hat{\beta}}-y_i)\bm g^\top \bm Z_i\bm g-\sum_{j\in \Gamma^\star}\mu g_j^2\\
	=&{\frac{2}{n}\sum_{i=1}^{n}(\bm{\hat \beta}_1^{o\top }\bm Z_i^{11}\bm g_1)^2+\frac{1}{n}\sum_{i=1}^{n}(\bm{\hat\beta} _1^o+\bm\beta_1^\star)^\top \bm Z_i^{11}(\bm{\hat\beta}_1^o-\bm \beta_1^\star)g_1^\top \bm Z_i^{11}\bm g_1}
	\\&\qquad-\frac{1}{n}\sum_{i=1}^{n}\bm g_1^\top \bm Z_i^{11}\bm g_1\varepsilon_i-\sum_{j\in \Gamma^\star}\mu g_j^2
	\\
	\geq&2c_3\|\bm{\hat\beta}_1^o\|^2\|\bm g_1\|^2-c_4\|\bm{\hat\beta}_1^o+\bm \beta_1^\star\|\|\bm{\hat\beta}_1^o-\bm \beta_1^\star\|\|\bm g_1\|^2
	\\
	&\qquad\qquad- 5\sigma\sqrt{c_4}\sqrt{\frac{s\ln (1+2n)}{n}}\|\bm g_1\|^2-\mu \|\bm g_1\|^2
	\\
	\geq&2c_3\|\bm{\hat\beta}_1^o\|^2\|\bm g_1\|^2-c_4\|\bm{\hat\beta}_1^o+\bm \beta_1^\star\|\|\bm{\hat\beta}_1^o-\bm \beta_1^\star\|\|\bm g_1\|^2-c_3\|\bm \beta_1^\star \|^2\|\bm g_1\|^2-\mu \|\bm g_1\|^2
	\\ \geq& \Big(2c_3(||\bm\beta_1^\star\|^2+\|\bm{\hat\beta}_1^o-\bm \beta_1^\star\|^2-2||\bm \beta_1^\star\|\|\bm{\hat\beta}_1^o-\bm \beta_1^\star\|)\Big)\|g_1\|^2
	\\&	-\Big(c_4(2\|\bm \beta_1^\star\|+\|{\hat\beta}_1^o-\bm\beta_1^\star\|)\|\ \bm{\hat\beta}_1^o-\bm \beta_1^\star\|+ 5\sigma\sqrt{c_4}\sqrt{\frac{s\ln (1+2n)}{n}}+\mu\Big)\|\bm g_1\|^2
	\\ \geq & \Big(2c_3\|\bm \beta_1^\star\|^2-4c_3\|\bm \beta_1 ^\star \|r_n-c_4r_n^2-2c_4\|\bm \beta_1^\star\| r_n-5\sigma\sqrt{c_4}\sqrt{\frac{s\ln (1+2n)}{n}}-\mu\Big)\|\bm g_1\|^2
	\\ \geq & \Big(2c_3c_1^2s-c_4r_n^2-(4c_3+2c_4)c_2\sqrt{s}r_n-5\sigma\sqrt{c_4}\sqrt{\frac{s\ln (1+2n)}{n}}-\mu\Big)\|\bm g_1\|^2.
\end{split}
\end{equation*}
\begin{equation*}
\begin{split} \geq & \Big(2c_3c_1^2s-(4c_3+3c_4)c_2\sqrt{s}r_n-5\sigma\sqrt{c_4}\sqrt{\frac{s\ln (1+2n)}{n}}-\mu\Big)\|\bm g_1\|^2
	\\ \geq & \frac{1}{2}\big(2c_3c_1^2s-(4c_3+3c_4)c_2\sqrt{s}r_n-5\sigma\sqrt{c_4}\sqrt{\frac{s\ln (1+2n)}{n}}-\mu\big)\|\bm g_1\|^2.
\end{split}
\end{equation*}
where the fifth inequality follows from $r_n\leq c_1\sqrt{s}$ and $c_1\leq c_2$.
Combing the  above inequality and the condition $\mu \leq c_1^2c_3s$ and  (\ref{lambdan_con}), 	
we  get (\ref{second-order-con2})	under the event $E^1$.


Overall, we prove that $\bm{\hat\beta}$ is a local minimizer under the event $E^1\cap E^2 \cap E^3$.
By Lemmas \ref {lemma 1}, \ref{prob-E2} and \ref{prob-E2}, we complete the proof.\qed

\section{Proofs of Propositions 1 and 2}
\setcounter{equation}{0}

\noindent{\bf Proof of Proposition 1} 
Based on the similar proof method to that of Theorem 5 in \citet{Fan2018}, it is easy to prove equation (3.6) and thus it is  omitted. We mainly prove equation (3.7).

For any $\tau >0$, define the following auxiliary function
\begin{equation}\label{S2.1}
G_\tau (\bm\beta|\bm\gamma):=L(\bm\gamma)+\langle  \nabla L(\bm\gamma),\bm\beta-\bm\gamma\rangle +\frac{1}{2\tau}\|\bm\beta-\bm\gamma\|^2+P_{\lambda}(\bm\gamma)+\sum_{j=1}^d w_{\lambda,j}(|\beta_j|-|\gamma_j|)
\end{equation}
where 
\begin{equation*}
w_{\lambda,j}=\begin{cases}
\max(p_{\lambda}^{'}(|\gamma_j|),\epsilon_1), & \text{if } \gamma_j\neq 0,\\
\lambda \varrho, & \text{if } \gamma_j=0.
\end{cases}
\end{equation*}

It is clear that minimizing \eqref{S2.1} with respect to $\bm{\beta}$ is equivalent to the following minimization problem
\begin{equation}\label{S2.2}
\min_{\bm\beta \in \mathbb{R}^d} \frac{1}{2} \|\bm\beta-(\bm\gamma+\tau \nabla L(\bm\gamma))\|^2+\tau \sum_{j=1}^d w_{\lambda,j}|\beta_j|.
\end{equation}

For any $r >0$, let $B_r=\{\bm\beta \in \mathbb{R}^d : \|\bm\beta\|\leq r\}$ and $\hat L=\sup_{\bm\beta \in B_r} \|\nabla^2 L(\bm\beta)\|.$ Then for any $\tau \in (0,{\hat L}^{-1}]$ and $\bm\beta, \bm\gamma \in B_r$, we have 
\begin{equation*}
\begin{split}
F(\bm\beta)&=L(\bm\gamma)+\langle  \nabla L(\bm\gamma),\bm\beta-\bm\gamma\rangle+\frac{1}{2}(\bm\beta-\bm\gamma)^\top \nabla^2 L(\bm\xi)(\bm\beta-\bm\gamma) +P_{\lambda}(\bm\beta)
\\& \leq L(\bm\gamma)+\langle  \nabla L(\bm\gamma),\bm\beta-\bm\gamma\rangle+\frac{1}{2}(\bm\beta-\bm\gamma)^\top \nabla^2 L(\bm\xi)(\bm\beta-\bm\gamma)
\\&+P_{\lambda}(\bm\gamma)+\sum_{j=1}^{n} v_j(|\beta_j|-|\gamma_j|)
\\& \leq L(\bm\gamma)+\langle  \nabla L(\bm\gamma),\bm\beta-\bm\gamma\rangle+\frac{1}{2}(\bm\beta-\bm\gamma)^\top \nabla^2 L(\bm\xi)(\bm\beta-\bm\gamma)
\\&+P_{\lambda}(\bm\gamma)+\sum_{\gamma_j\neq0}\max(p_{\lambda}^{'}(|\gamma_j|),\epsilon_1)(|\beta_j|-|\gamma_j|)+\sum_{\gamma_j=0}\lambda \varrho(|\beta_j|-|\gamma_j|)
\end{split}
\end{equation*}
\begin{equation*}
\begin{split}
& \leq G_\tau(\bm\beta|\bm\gamma)+\frac{1}{2}(\bm\beta-\bm\gamma)^\top \nabla^2 L(\bm\xi)(\bm\beta-\bm\gamma)-\frac{1}{2\tau}\|\bm\beta-\bm\gamma\|^2
\\& \leq G_\tau(\bm\beta|\bm\gamma)+\frac{\hat L}{2}\|\bm\beta-\bm\gamma\|^2-\frac{1}{2\tau}\|\bm\beta-\bm\gamma\|^2
\\& \leq G_\tau(\bm\beta|\bm\gamma).
\end{split}
\end{equation*}
where $v_j\in \partial p_{\lambda}(|\gamma_j|)$ and $\bm\xi=\bm\gamma+\alpha(\bm\beta-\bm\gamma)$ for some $\alpha \in (0,1)$. The first  inequality follows from the concavity of $p_{\lambda}$, the second inequality follows from the upper bound of $\partial p_{\lambda}$ and the third inequality follows from $\|\bm\xi\|\leq r$. 

Assuming $\bm{\hat\beta}\in \arg\min\limits_{\bm\beta\in \mathbb{R}^d}F(\bm\beta)$ and $\bm{\tilde\beta}\in \arg\min\limits_{\bm\beta\in \mathbb{R}^d}G(\bm\beta|\bm{\hat\beta})$, then we get
\begin{equation*}
G_\tau(\bm{\tilde\beta}|\bm{\hat\beta})\leq G_\tau (\bm{\hat\beta}|\bm{\hat\beta})=F(\bm{\hat\beta})\leq F(\bm{\tilde\beta}) \leq G_\tau(\bm{\tilde\beta}|\bm{\hat\beta}),
\end{equation*}
hence $\bm{\hat\beta}\in \arg\min\limits_{\bm\beta\in \mathbb{R}^d}G_\tau(\bm\beta|\bm{\hat\beta})$. It means that $\bm{\hat\beta}$ is also a minimizer of the problem \eqref{S2.2} with $\bm\gamma=\bm{\hat\beta}$.

It is clear that problem \eqref{S2.2} can be solved using the soft-thresholding operator $\mathcal{S}$, therefore we can get (\ref{fixed2}).
\qed 

To prove Proposition \ref{alg-convergence}, we also need the following lemma.
\begin{lemma}\label{armijo}
Let $g_k=\|\nabla{}\ell(\beta^k)\|_2$, $G_k=\sup_{\beta\in{}B_k}\|\nabla^2\ell(\beta)\|_2$ where $B_k=\{\beta\in\mathbb{R}^p:\|\beta\|_2\leq\|\beta^k\|_2+g_k\}.$ For any $\delta>0, \alpha_0,\gamma_1\in(0,1)$, define
$$j_k=\left\{
\begin{array}{lll}
0, & \mbox{if}~\gamma_1(G_k+\delta)\leq1;\\
-\mbox{[}\log_{\gamma_0}\gamma_1(G_k+\delta)\mbox{]}+1, & \mbox{otherwise}.
\end{array}
\right.$$
Then (\ref{armijo1}) holds, and there exists a nonnegative integer $\bar{j}$ such that $\tau_k\in[\gamma_1\gamma_0^{\bar{j}},\gamma_1].$
\end{lemma}
\noindent{\bf Proof sketch} 
The proof follows the same structure as that of Lemmas 8 and 9 in \cite{Fan2018}. The core function \( L \) is identical. Although the regularizer considered here (a weakly convex concave regularizer or a weighted \( \ell_1 \) penalty) differs in form from the \( \ell_q(0<q<1) \) regularizer in \cite{Fan2018}, the argument relies only on the coercivity of the regularizer and the nonexpansiveness of the proximal operator, both of which hold in our setting. Consequently, the entire line of reasoning applies directly, and detailed calculations are omitted. \qed

\noindent{\bf Proof of Proposition \ref{alg-convergence}}  Regarding Algorithm 1, based on the similar proof method to that of Theorem S3.1 in \cite{Fan2018}, it is easy to prove it and so is  omitted. The conclusion for Algorithm 2 is proved below, where the proof of (i) follows the approach of Proposition 3.2 in \cite{bai2024avoiding}.
For any $\tau >0$, define the following auxiliary function
\begin{equation}\label{S2.1}
G^k_\tau (\bm\beta,\bm{\beta}^k):=L(\bm{\beta}^k)+\langle  \nabla L(\bm{\beta}^k),\bm\beta-\bm{\beta}^k\rangle +\frac{1}{2\tau}\|\bm\beta-\bm{\beta}^k\|^2+P_{\lambda}(\bm\gamma)
+\sum_{j=1}^d w_{\lambda,j}^k(|\beta_j|-|\beta_j^k|).
\end{equation}
where 
\begin{equation*}
w_{\lambda,j}^k=\begin{cases}
\max(p_{\lambda}^{'}(|\beta_j^k|),\epsilon_1), & \text{if } \beta_j^k\neq 0,\\
\lambda \varrho, & \text{if } \beta_j^k=0.
\end{cases}
\end{equation*}
Notice that the iteration $\bm\beta^{k+1}=\mathcal{S}(\bm\beta^k-\tau_k\nabla  L(\bm\beta^k),{\lambda}\tau_k \bm{w}^k)$ is equivalent to
$$\bm\beta^{k+1}=\arg\min\frac{1}{2}\|\bm\beta-\big(\bm\beta^k-\tau_k\nabla  L(\bm\beta^k)\big)\|^2+\lambda\tau_k\|\bm{w}^k\circ\bm{\beta}\|_1$$
Hence the sequence $\{\bm\beta^{k}\}$ is bounded because every entry of $\bm{w}^k$ is positive, $\lambda>0$ and $\tau_k\in[\gamma_1\gamma_0^{\bar{j}},\gamma_1]$ from Lemma \ref{armijo}. 
Then there exists a positive constant $G$ such that $G=\sup\|\nabla^2\ell(\bm\beta^k)\|_2$.

(i)~From the convavity of $p_{\lambda}^{'}(t)>0$ for any $t>0$, we conclude that 
\begin{equation}\label{prop-concave-pena}
\begin{split}
&\sum_{j=1}^{d}p_{\lambda}(|\beta_j^{k+1}|)-\sum_{j=1}^{d}p_{\lambda}(|\beta_j^{k}|)\\
\leq& 
\sum_{\beta_j^{k}\neq0}p_{\lambda}^{'}(|\beta_j^{k}|)(|\beta_j^k|-|\beta_j^{k+1}|)+\sum_{\beta_j^{k}=0}\xi_j^k(|\beta_j^k|-|\beta_j^{k+1}|)\nonumber
\end{split}
\end{equation}
\begin{equation}\label{prop-concave-pena}
\begin{split}
\leq& 
\sum_{\beta_j^{k}\neq0}p_{\lambda}^{'}(|\beta_j^{k}|)(|\beta_j^k|-|\beta_j^{k+1}|)+\sum_{\beta_j^{k}=0}\lambda \varrho(|\beta_j^k|-|\beta_j^{k+1}|)\\
=& \sum_{j=1}^{d}w_{\lambda,j}^k(|\beta_j^{k+1}|-|\beta_j^k|)
\end{split}
\end{equation}
where the second inequality follows from $\lim\limits_{t\rightarrow0^+}p_{\lambda}^{'}(t)=\lambda \varrho$ and $\varrho>0$.
Additionally, from the Taylor expansion, it follows that
\begin{equation}\label{S2.4}
L(\bm\beta^k)-L(\bm\beta^{k+1})\geq \langle \nabla L(\bm\beta^k),\bm\beta^k-\bm\beta^{k+1} \rangle -\frac{G}{2}\|\bm\beta^k-\bm\beta^{k+1}\|^2.
\end{equation}

Combining  \eqref{prop-concave-pena} and \eqref{S2.4}, we can get
\begin{equation}\label{S2.5}
\begin{split}
&F(\bm\beta ^k)-F(\bm\beta^{k+1})\\
=&	L(\bm\beta^k)+\sum_{j=1}^{d}p_{\lambda}(|\beta_j^{k}|)-L(\bm\beta^{k+1})-\sum_{j=1}^{d}p_{{\lambda}}(|\beta_j^{k+1}|)
\\ \geq& \langle \nabla L(\bm\beta^k),\bm\beta^k-\bm\beta^{k+1} \rangle -\frac{\hat L}{2}\|\bm\beta^k-\bm\beta^{k+1}\|^2 
+\sum_{j=1}^{d}w_{\lambda,j}^k(|\beta_j^{k}|-|\beta_j^{k+1}|)
\end{split}
\end{equation}
For each subproblem, $\bm\beta^{k+1}$ staisfies the optimal condition
\begin{equation*}
\bm{0}\in\nabla L(\bm\beta^k)+\frac{1}{\tau_k}(\bm\beta ^{k+1}-\bm\beta^k)+\bm{w}_\lambda^k \circ\partial\|\bm{\beta}^{k+1}\|_1,
\end{equation*}
which implies that there exists an vector $\bm\zeta^{k+1}=(\zeta_1^{k+1},\cdots,\zeta_d^{k+1})^\top \in \partial \|\bm\beta^{k+1}\|_1$ such that 
\begin{equation}\label{Gtau-opt-cond}
\nabla L(\bm\beta^k)+\frac{1}{\tau_k}(\bm\beta ^{k+1}-\bm\beta^k)+\bm{w}_\lambda^k \circ\bm\zeta^{k+1}=	\bm{0},
\end{equation} 
where $\bm{w}_\lambda^k=(w_{\lambda,1}^k,\cdots,w_{\lambda,d}^k)^\top$. Therefore, we obtain 
\begin{eqnarray}
\begin{aligned}
&G^k_{\tau_k} (\bm\beta^{k},\bm\beta^k)-G^k_{\tau_k} (\bm\beta^{k+1},\bm\beta^k)\\
=& \langle \nabla L(\bm\beta^k),\bm\beta^k-\bm\beta^{k+1} \rangle -\frac{1}{2\tau_k}\|\bm\beta^k-\bm\beta^{k+1}\|^2 +\sum_{j=1}^{d}w_j^k(|\beta_j^{k}|-|\beta_j^{k+1}|)
\\ \geq & \langle \nabla L(\bm\beta^k),\bm\beta^k-\bm\beta^{k+1} \rangle  +\sum_{j=1}^{d}\zeta_j^{k+1}w_j^k(\beta_j^{k}-\beta_j^{k+1})-\frac{1}{2\tau_k}\|\bm\beta^k-\bm\beta^{k+1}\|^2
\\ = & [\nabla L(\bm\beta^k)+\frac{1}{\tau}(\bm\beta^{k+1}-\bm\beta^k)+\sum_{j=1}^{d} w_j^k\zeta_j^{k+1}]^\top(\bm\beta^k-\bm\beta^{k+1})
+\frac{1}{2\tau_k}\|\bm\beta^k-\bm\beta^{k+1}\|^2\\
=&\frac{1}{2\tau_k}\|\bm\beta^k-\bm\beta^{k+1}\|^2,
\end{aligned}\nonumber
\end{eqnarray}
where the inequality follows from the convexity of $|\cdot|$ and $w_{\lambda,j}^k\geq0$, and the last equality derives from (\ref{Gtau-opt-cond}).
Combining this and \eqref{S2.5}, we get 
\begin{eqnarray}\label{S2.7}
\begin{aligned}
&F(\bm\beta ^k)-F(\bm\beta^{k+1})\\ \geq & G_{\tau_k}^k(\bm\beta^{k},\bm\beta^k)-G_{\tau_k}^k(\bm\beta^{k+1},\bm\beta^k)-\frac{G}{2}\|\bm\beta^k-\bm\beta^{k+1}\|^2+\frac{1}{2\tau_k}\|\bm\beta^k-\bm\beta^{k+1}\|^2 \\\geq & 
\left(\frac{1}{\tau_k}-\frac{G}{2}\right)\|\bm\beta^k-\bm\beta^{k+1}\|^2
\\ \geq& 0.
\end{aligned}
\end{eqnarray}

(ii)~From the definition of $\bm\beta^{k+1}$ and \eqref{S2.7}, we have 
\begin{equation*}
\sum_{k=0}^{K}[F(\bm\beta^k)-F(\bm\beta^{k+1})]=F(\bm\beta^0)-F(\bm\beta^{K+1})\geq \left(\frac{1}{\tau_k}-\frac{G}{2}\right)\sum_{k=0}^{K}\|\bm\beta^k-\bm\beta^{k+1}\|^2.
\end{equation*}
Hence, $\sum_{k=0}^{\infty} \|\bm\beta^k - \bm\beta^{k+1}\|^{2} < \infty$, which implies $\lim\limits_{k\rightarrow \infty}\|\bm\beta^k - \bm\beta^{k+1}\|^{2}=0$.

(iii)~Because $\{\bm\beta^k\}$ and $\{\tau_k\}$ have convergent subsequences, without loss of generality, assume that
\begin{equation}\label{S2.8}
\bm\beta^k \to \bm\beta^\diamond \quad \text{and} \quad \tau_k \to \tau^\diamond\quad\text{as } k \to \infty,
\end{equation}
and denote
$$G_{\tau}^\diamond(\bm\beta, \bm\beta^\diamond)=L(\bm\beta^\diamond)+\langle  \nabla L(\bm\beta^\diamond),\bm\beta-\bm\beta^\diamond\rangle +\frac{1}{2\tau}\|\bm\beta-\bm\beta^\diamond\|^2+P_{\lambda}(\bm\gamma)+\sum_{j=1}^d w_{j}^\diamond(|\beta_j|-|\beta_j^\diamond|)$$
with
\begin{equation*}
w_j^\diamond=\begin{cases}
\max(p_{\lambda}^{'}(|\beta_j^\diamond|),\epsilon_1), & \text{if } \beta_j^\diamond\neq 0,\\
\lambda \varrho, & \text{if } \beta_j^\diamond=0.
\end{cases}
\end{equation*}
It is known that the iteration $\bm\beta^{k+1}=\mathcal{S}(\bm\beta^k-\tau_k\nabla  L(\bm\beta^k),{\lambda}\tau_k \bm{w}^k).$ Then for any $\bm\beta \in \mathbb{R}^d$,
\begin{equation*}
G^k_{\tau_k} (\bm\beta^{k+1},\bm\beta^k)\leq G^k_{\tau_k} (\bm\beta,\bm\beta^k).
\end{equation*}
Combing the continuousness of the function $	G^k_\tau(\cdot,\bm\beta^k)$ and the limit \eqref{S2.8}, we can obtain 
\begin{equation*}
G_{\tau^\diamond}^\diamond(\bm\beta^\diamond, \bm\beta^\diamond)
= \lim_{k \to \infty} G_{\tau_k}^k(\bm\beta^{k+1}, \bm\beta^k)\leq \lim_{k \to \infty} G_{\tau_k}^k(\bm\beta, \bm\beta^k)
= G_{\tau^\diamond}^\diamond(\bm\beta, \bm\beta^\diamond).
\end{equation*}
Based on the above inequality, we get
\begin{equation*}
\bm\beta^\diamond \in \arg\min_{\bm\beta \in \mathbb{R}^d} G_{\tau^\diamond}^\diamond(\bm\beta, \bm\beta^\diamond).
\end{equation*}
According to the Proposition 1, it means that $\bm\beta^\diamond $ is also satisfies 
\begin{equation}
\bm\beta=\mathcal{S}(\bm\beta-\tau\nabla L(\bm\beta),\tau \bm{w}_\lambda).\nonumber
\end{equation}
Then, we complete the proof.
\qed

\newpage

\section{Detailed numerical results}




\begin{table}[htbp]
	\centering
	\caption{\text{RelErr} for $p = 128$, $s = 10$, $\sigma = 10^{-2}$.}
	\label{tab:error_p128_s10_s1e-2}
	\scriptsize
	\renewcommand{\arraystretch}{1.2}
	\resizebox{\linewidth}{!}{%
\begin{tabular}{c|ccccccc}
			\toprule
			$n/p$  & $\ell_1$ & $\ell_{1/2}$ & SCAD & MCP & TL1 & LOG & EXP \\
			\midrule
			0.1 & 1.10 (0.11) & 1.11 (0.11) & 1.12 (0.12) & 1.11 (0.12) & 1.20 (0.15) & 1.20 (0.14) & 1.17 (0.12) \\
			0.2 & 1.06 (0.17) & 1.08 (0.14) & 1.06 (0.19) & 1.05 (0.18) & 1.10 (0.36) & 1.13 (0.25) & 1.15 (0.24) \\
			0.3 & 8.40e-1 (3.65e-1) & 9.54e-1 (2.93e-1) & 8.14e-1 (3.80e-1) & 8.45e-1 (3.29e-1) & 6.39e-1 (6.16e-1) & 7.46e-1 (5.65e-1) & 8.04e-1 (5.59e-1) \\
			0.4 & 4.93e-1 (5.40e-1) & 7.02e-1 (4.76e-1) & 4.60e-1 (4.89e-1) & 4.91e-1 (4.68e-1) & 3.45e-1 (5.63e-1) & 4.12e-1 (5.59e-1) & 4.40e-1 (5.98e-1) \\
			0.5 & 2.30e-1 (4.80e-1) & 3.88e-1 (5.37e-1) & 2.49e-1 (4.66e-1) & 2.34e-1 (4.43e-1) & 1.66e-1 (4.33e-1) & 2.42e-1 (4.87e-1) & 2.70e-1 (5.18e-1) \\
			0.6 & 9.54e-2 (3.25e-1) & 2.11e-1 (4.33e-1) & 9.74e-2 (3.13e-1) & 1.14e-1 (3.29e-1) & 7.39e-2 (2.95e-1) & 8.59e-2 (3.06e-1) & 1.10e-1 (3.54e-1) \\
			0.7 & 1.20e-2 (1.18e-1) & 8.37e-2 (2.89e-1) & 1.81e-2 (1.27e-1) & 3.52e-2 (1.74e-1) & 1.21e-2 (1.20e-1) & 3.50e-2 (1.92e-1) & 1.18e-2 (1.17e-1) \\
			0.8 & 1.40e-2 (1.39e-1) & 2.38e-2 (1.68e-1) & 1.23e-2 (1.23e-1) & 2.21e-2 (1.55e-1) & 1.29e-2 (1.29e-1) & 2.49e-2 (1.74e-1) & 2.55e-2 (1.79e-1) \\
			0.9 & 1.19e-2 (1.18e-1) & 1.25e-2 (1.23e-1) & 9.82e-3 (9.77e-2) & 1.01e-2 (1.00e-1) & 7.39e-5 (3.12e-5) & 1.12e-2 (1.11e-1) & 1.25e-2 (1.24e-1) \\
			1.0 & 9.33e-5 (3.14e-5) & 3.07e-4 (1.06e-3) & 4.89e-5 (1.72e-5) & 4.60e-5 (1.70e-5) & 7.16e-5 (2.60e-5) & 1.97e-4 (1.04e-3) & 6.04e-5 (2.24e-5) \\
			\bottomrule
		\end{tabular}%
	}
\end{table}

\begin{table}[htbp]
	\centering
	\caption{Computation time (seconds) for $p = 128$, $s = 10$, $\sigma = 10^{-2}$.}
	\label{tab:time_p128_s10_s1e-2}
	\scriptsize
	\renewcommand{\arraystretch}{1.2}
	\resizebox{\linewidth}{!}{%
		\begin{tabular}{c|ccccccc}
			\toprule
			$n/p$ & $\ell_1$ & $\ell_{1/2}$ & SCAD & MCP & TL1 & LOG & EXP \\
			\midrule
			0.1 & 1.02e-2 (2.74e-3) & 2.71e-2 (1.68e-2) & 5.26e-3 (3.94e-3) & 4.29e-3 (1.25e-3) & 1.69e-1 (4.94e-2) & 7.49e-2 (1.46e-2) & 1.21e-1 (6.56e-2) \\
			0.2 & 2.37e-2 (6.71e-3) & 6.37e-2 (2.24e-2) & 4.56e-3 (1.89e-3) & 4.75e-3 (1.25e-3) & 2.14e-1 (5.78e-2) & 9.10e-2 (1.60e-2) & 2.59e-1 (1.40e-1) \\
			0.3 & 7.10e-2 (3.65e-2) & 9.25e-2 (5.48e-2) & 6.69e-3 (3.37e-3) & 6.82e-3 (4.27e-3) & 1.55e-1 (9.05e-2) & 9.94e-2 (6.63e-2) & 4.66e-1 (2.69e-1) \\
			0.4 & 2.61e-1 (1.74e-2) & 1.33e-1 (8.08e-2) & 2.61e-2 (1.54e-2) & 2.39e-2 (1.10e-2) & 1.78e-1 (1.37e-1) & 1.70e-1 (1.29e-1) & 7.44e-1 (7.38e-1) \\
			0.5 & 2.82e-1 (2.76e-2) & 1.10e-1 (8.11e-2) & 3.28e-2 (1.25e-2) & 2.77e-2 (9.24e-3) & 1.45e-1 (1.41e-1) & 1.38e-1 (1.27e-1) & 6.10e-1 (8.68e-1) \\
			0.6 & 2.24e-1 (6.83e-2) & 7.46e-2 (7.64e-2) & 3.25e-2 (8.26e-3) & 2.81e-2 (6.42e-3) & 8.85e-2 (4.03e-2) & 8.79e-2 (1.06e-1) & 3.85e-1 (7.96e-1) \\
			0.7 & 1.53e-1 (4.62e-2) & 6.26e-2 (6.76e-2) & 3.22e-2 (6.15e-3) & 2.90e-2 (5.20e-3) & 7.65e-2 (2.59e-2) & 7.81e-2 (9.37e-2) & 1.79e-1 (3.63e-1) \\
			0.8 & 1.33e-1 (3.71e-2) & 5.32e-2 (5.26e-2) & 3.94e-2 (7.68e-3) & 3.59e-2 (7.19e-3) & 9.22e-2 (3.71e-2) & 8.08e-2 (9.95e-2) & 1.97e-1 (4.97e-1) \\
			0.9 & 1.22e-1 (3.98e-2) & 4.52e-2 (1.51e-2) & 4.10e-2 (7.24e-3) & 4.08e-2 (1.56e-2) & 8.80e-2 (2.09e-2) & 9.06e-2 (1.06e-1) & 1.76e-1 (4.46e-1) \\
			1.0 & 1.07e-1 (1.71e-2) & 4.45e-2 (1.32e-2) & 4.51e-2 (1.97e-2) & 4.06e-2 (1.01e-2) & 8.49e-2 (1.40e-2) & 9.21e-2 (1.24e-1) & 1.32e-1 (1.00e-1) \\
			\bottomrule
		\end{tabular}%
	}
\end{table}

\begin{table}[htbp]
	\centering
	\caption{Support recovery metrics for $p = 128$, $s = 10$, $\sigma = 10^{-2}$.}
	\label{tab:support_p128_s10_s1e-2}
	\scriptsize
	\renewcommand{\arraystretch}{1.2}
	\resizebox{\linewidth}{!}{%
		\begin{tabular}{c|c|ccccccc}
			\toprule
			\multirow{2}{*}{$n/p$} & \multirow{2}{*}{Metric} & \multicolumn{7}{c}{Penalty} \\
			\cmidrule{3-9}
			& & $\ell_1$ & $\ell_{1/2}$ & SCAD & MCP & TL1 & LOG & EXP \\
			\midrule
			0.1 & TPR & 4.81e-1 (1.51e-1) & 2.72e-1 (1.37e-1) & 2.38e-1 (1.54e-1) & 2.40e-1 (1.63e-1) & 1.70e-1 (1.30e-1) & 9.70e-2 (9.37e-2) & 1.81e-1 (1.24e-1) \\
			& FPR & 4.68e-1 (4.63e-2) & 2.30e-1 (5.82e-2) & 1.95e-1 (6.49e-2) & 1.96e-1 (8.75e-2) & 1.39e-1 (3.23e-2) & 5.77e-2 (1.92e-2) & 1.58e-1 (3.41e-2) \\
			& F1  & 1.37e-1 (4.13e-2) & 1.35e-1 (6.24e-2) & 1.30e-1 (7.49e-2) & 1.30e-1 (7.86e-2) & 1.20e-1 (9.07e-2) & 1.07e-1 (1.01e-1) & 1.19e-1 (7.89e-2) \\
			\midrule
			0.2 & TPR & 5.99e-1 (1.61e-1) & 4.15e-1 (1.51e-1) & 3.32e-1 (1.78e-1) & 3.28e-1 (1.78e-1) & 2.70e-1 (2.34e-1) & 1.94e-1 (1.44e-1) & 2.81e-1 (1.74e-1) \\
			& FPR & 5.27e-1 (4.73e-2) & 2.97e-1 (5.02e-2) & 2.24e-1 (7.13e-2) & 2.21e-1 (8.08e-2) & 1.38e-1 (4.74e-2) & 8.72e-2 (2.79e-2) & 1.89e-1 (5.06e-2) \\
			& F1  & 1.53e-1 (4.02e-2) & 1.69e-1 (5.94e-2) & 1.66e-1 (8.15e-2) & 1.64e-1 (7.72e-2) & 1.91e-1 (1.70e-1) & 1.79e-1 (1.44e-1) & 1.64e-1 (1.20e-1) \\
\bottomrule
		\end{tabular}%
	}
\end{table}

\begin{table}[htbp]
	\centering
	\caption{Relative bias for large coefficients and estimated sparsity for $p = 128$, $s = 10$, $\sigma = 10^{-2}$ (true sparsity = 10). Sparsity values are shown as mean (SD) in decimal format.}
	\label{tab:bias_p128_s10_s1e-2}
	\scriptsize
	\renewcommand{\arraystretch}{1.2}
	\resizebox{\linewidth}{!}{%
		\begin{tabular}{c|c|ccccccc}
			\toprule
			\multirow{2}{*}{$n/p$} & \multirow{2}{*}{Metric} & \multicolumn{7}{c}{Penalty} \\
			\cmidrule{3-9}
			& & $\ell_1$ & $\ell_{1/2}$ & SCAD & MCP & TL1 & LOG & EXP \\
			\midrule
			0.1 & Bias & 9.10e-1 (9.78e-2) & 9.10e-1 (9.35e-2) & 9.17e-1 (9.36e-2) & 9.18e-1 (9.20e-2) & 9.31e-1 (9.25e-2) & 9.27e-1 (9.81e-2) & 9.17e-1 (9.76e-2) \\
			& Sparsity & 60.08 (5.81) & 29.84 (7.37) & 25.33 (8.43) & 25.58 (11.25) & 18.04 (3.90) & 7.78 (2.30) & 20.41 (4.12) \\
			\midrule
			0.2 & Bias & 8.29e-1 (1.48e-1) & 8.29e-1 (1.37e-1) & 8.46e-1 (1.57e-1) & 8.44e-1 (1.49e-1) & 8.08e-1 (2.63e-1) & 8.33e-1 (1.89e-1) & 8.30e-1 (1.88e-1) \\
			& Sparsity & 68.13 (5.74) & 39.15 (6.17) & 29.78 (9.16) & 29.33 (10.50) & 19.00 (5.09) & 12.23 (2.88) & 25.07 (5.86) \\
			\midrule
			0.3 & Bias & 6.12e-1 (3.01e-1) & 6.73e-1 (2.54e-1) & 6.31e-1 (3.12e-1) & 6.62e-1 (2.74e-1) & 5.57e-1 (1.09e+0) & 5.30e-1 (4.14e-1) & 5.45e-1 (3.94e-1) \\
			& Sparsity & 69.52 (6.05) & 40.24 (9.09) & 31.66 (11.34) & 30.13 (12.62) & 18.05 (3.30) & 12.58 (4.44) & 22.15 (5.90) \\
\bottomrule
		\end{tabular}%
	}
\end{table}


\begin{table}[htbp]
	\centering
	\caption{\text{RelErr} for $p = 128$, $s = 10$, $\sigma = 5\times10^{-2}$.}
	\label{tab:error_p128_s10_s5e-2}
	\scriptsize
	\renewcommand{\arraystretch}{1.2}
	\resizebox{\linewidth}{!}{%
		\begin{tabular}{c|ccccccc}
			\toprule
			$n/p$  & $\ell_1$ & $\ell_{1/2}$ & SCAD & MCP & TL1 & LOG & EXP \\
			\midrule
			0.1 & 1.11 (1.03e-1) & 1.12 (1.00e-1) & 1.13 (1.09e-1) & 1.12 (1.11e-1) & 1.20 (1.25e-1) & 1.21 (1.28e-1) & 1.18 (1.11e-1) \\
			0.2 & 1.05 (1.90e-1) & 1.07 (1.60e-1) & 1.04 (2.10e-1) & 1.05 (1.86e-1) & 1.04 (4.53e-1) & 1.10 (3.09e-1) & 1.11 (2.65e-1) \\
			0.3 & 9.43e-1 (3.36e-1) & 1.01 (2.85e-1) & 8.95e-1 (3.46e-1) & 9.16e-1 (3.08e-1) & 7.15e-1 (6.20e-1) & 8.65e-1 (5.28e-1) & 9.36e-1 (5.18e-1) \\
			0.4 & 4.95e-1 (5.28e-1) & 7.43e-1 (4.81e-1) & 4.60e-1 (5.14e-1) & 5.65e-1 (4.83e-1) & 3.60e-1 (5.59e-1) & 4.56e-1 (5.77e-1) & 4.92e-1 (6.01e-1) \\
			0.5 & 2.14e-1 (4.45e-1) & 4.04e-1 (5.17e-1) & 2.29e-1 (4.22e-1) & 2.87e-1 (4.33e-1) & 1.42e-1 (3.88e-1) & 2.00e-1 (4.37e-1) & 2.53e-1 (4.84e-1) \\
			0.6 & 1.00e-1 (3.43e-1) & 1.67e-1 (3.94e-1) & 9.37e-2 (3.07e-1) & 1.12e-1 (3.31e-1) & 8.69e-2 (3.19e-1) & 9.24e-2 (3.15e-1) & 1.03e-1 (3.34e-1) \\
			0.7 & 1.38e-2 (1.32e-1) & 3.40e-2 (1.92e-1) & 1.14e-2 (1.11e-1) & 2.33e-2 (1.62e-1) & 2.29e-2 (1.58e-1) & 1.20e-2 (1.05e-1) & 2.54e-2 (1.76e-1) \\
			0.8 & 5.71e-4 (2.14e-4) & 9.45e-3 (8.94e-2) & 2.96e-4 (1.25e-4) & 2.84e-4 (1.23e-4) & 4.45e-4 (1.75e-4) & 8.97e-4 (2.06e-3) & 3.76e-4 (1.52e-4) \\
			0.9 & 4.92e-4 (1.63e-4) & 5.58e-4 (8.14e-4) & 2.56e-4 (1.06e-4) & 2.42e-4 (1.03e-4) & 3.83e-4 (1.45e-4) & 5.53e-4 (3.29e-4) & 3.22e-4 (1.26e-4) \\
			1.0 & 4.55e-4 (1.37e-4) & 3.43e-4 (4.74e-4) & 2.43e-4 (7.81e-5) & 2.29e-4 (7.53e-5) & 3.46e-4 (1.22e-4) & 4.31e-4 (2.00e-4) & 2.89e-4 (9.78e-5) \\
			\bottomrule
		\end{tabular}%
	}
\end{table}

\begin{table}[htbp]
	\centering
	\caption{Computation time (seconds) for $p = 128$, $s = 10$, $\sigma = 5\times10^{-2}$.}
	\label{tab:time_p128_s10_s5e-2}
	\scriptsize
	\renewcommand{\arraystretch}{1.2}
	\resizebox{\linewidth}{!}{%
		\begin{tabular}{c|ccccccc}
			\toprule
			$n/p$ & $\ell_1$ & $\ell_{1/2}$ & SCAD & MCP & TL1 & LOG & EXP \\
			\midrule
			0.1 & 9.64e-3 (3.51e-3) & 2.44e-2 (1.48e-2) & 5.18e-3 (3.54e-3) & 4.40e-3 (1.78e-3) & 1.68e-1 (5.43e-2) & 7.64e-2 (1.91e-2) & 1.23e-1 (6.09e-2) \\
			0.2 & 2.24e-2 (7.28e-3) & 5.94e-2 (2.20e-2) & 4.74e-3 (1.77e-3) & 4.76e-3 (8.76e-4) & 2.13e-1 (5.68e-2) & 9.69e-2 (1.59e-2) & 2.48e-1 (9.15e-2) \\
			0.3 & 6.26e-2 (1.97e-2) & 7.69e-2 (3.31e-2) & 6.34e-3 (3.71e-3) & 6.24e-3 (3.23e-3) & 1.43e-1 (7.94e-2) & 9.20e-2 (3.46e-2) & 4.88e-1 (2.73e-1) \\
			0.4 & 2.27e-1 (6.77e-2) & 1.14e-1 (7.67e-2) & 2.13e-2 (1.13e-2) & 1.78e-2 (9.01e-3) & 1.63e-1 (1.46e-1) & 1.63e-1 (1.24e-1) & 7.96e-1 (7.34e-1) \\
			0.5 & 2.82e-1 (3.66e-2) & 1.12e-1 (8.95e-2) & 2.72e-2 (8.83e-3) & 2.41e-2 (7.18e-3) & 1.25e-1 (1.09e-1) & 1.41e-1 (1.34e-1) & 6.28e-1 (8.97e-1) \\
			0.6 & 2.07e-1 (6.77e-2) & 7.37e-2 (7.15e-2) & 2.91e-2 (6.88e-3) & 2.66e-2 (5.30e-3) & 1.02e-1 (8.82e-2) & 9.80e-2 (9.68e-2) & 3.73e-1 (7.51e-1) \\
			0.7 & 1.49e-1 (4.21e-2) & 4.52e-2 (3.38e-2) & 2.90e-2 (6.20e-3) & 2.62e-2 (4.91e-3) & 7.64e-2 (3.72e-2) & 7.39e-2 (7.60e-2) & 1.81e-1 (4.57e-1) \\
			0.8 & 1.43e-1 (3.61e-2) & 4.35e-2 (1.32e-2) & 4.00e-2 (4.28e-2) & 3.27e-2 (8.47e-3) & 8.86e-2 (1.76e-2) & 8.16e-2 (6.47e-2) & 1.23e-1 (4.85e-2) \\
			0.9 & 1.19e-1 (2.18e-2) & 4.13e-2 (1.17e-2) & 3.71e-2 (7.24e-3) & 3.37e-2 (5.48e-3) & 8.54e-2 (1.52e-2) & 7.90e-2 (4.17e-2) & 1.14e-1 (2.01e-2) \\
			1.0 & 1.10e-1 (1.90e-2) & 3.99e-2 (8.52e-3) & 3.95e-2 (7.56e-3) & 3.68e-2 (6.30e-3) & 8.91e-2 (2.66e-2) & 7.06e-2 (3.83e-2) & 1.17e-1 (3.60e-2) \\
			\bottomrule
		\end{tabular}%
	}
\end{table}

\begin{table}[htbp]
	\centering
	\caption{Support recovery metrics for $p = 128$, $s = 10$, $\sigma = 5\times10^{-2}$.}
	\label{tab:support_p128_s10_s5e-2}
	\scriptsize
	\renewcommand{\arraystretch}{1.2}
	\resizebox{\linewidth}{!}{%
		\begin{tabular}{c|c|ccccccc}
			\toprule
			\multirow{2}{*}{$n/p$} & \multirow{2}{*}{Metric} & \multicolumn{7}{c}{Penalty} \\
			\cmidrule{3-9}
			& & $\ell_1$ & $\ell_{1/2}$ & SCAD & MCP & TL1 & LOG & EXP \\
			\midrule
			0.1 & TPR & 5.38e-1 (1.69e-1) & 2.83e-1 (1.56e-1) & 2.56e-1 (1.57e-1) & 2.40e-1 (1.59e-1) & 1.73e-1 (1.29e-1) & 9.70e-2 (9.79e-2) & 1.93e-1 (1.31e-1) \\
			& FPR & 4.69e-1 (4.93e-2) & 2.39e-1 (6.96e-2) & 1.97e-1 (6.72e-2) & 1.85e-1 (7.92e-2) & 1.44e-1 (2.92e-2) & 6.16e-2 (1.69e-2) & 1.56e-1 (3.30e-2) \\
			& F1  & 1.53e-1 (4.76e-2) & 1.39e-1 (7.32e-2) & 1.41e-1 (7.85e-2) & 1.37e-1 (8.21e-2) & 1.18e-1 (8.67e-2) & 1.03e-1 (9.33e-2) & 1.26e-1 (7.98e-2) \\
			\midrule
			0.2 & TPR & 6.14e-1 (1.56e-1) & 4.34e-1 (1.48e-1) & 3.54e-1 (1.65e-1) & 3.45e-1 (1.71e-1) & 3.04e-1 (2.67e-1) & 2.17e-1 (1.44e-1) & 3.11e-1 (1.69e-1) \\
			& FPR & 5.30e-1 (5.00e-2) & 3.02e-1 (5.51e-2) & 2.31e-1 (7.14e-2) & 2.27e-1 (9.19e-2) & 1.37e-1 (4.88e-2) & 8.67e-2 (3.00e-2) & 1.86e-1 (4.73e-2) \\
			& F1  & 1.57e-1 (4.06e-2) & 1.76e-1 (6.51e-2) & 1.75e-1 (8.32e-2) & 1.73e-1 (8.37e-2) & 2.15e-1 (2.04e-1) & 2.03e-1 (1.59e-1) & 1.83e-1 (1.19e-1) \\
\bottomrule
		\end{tabular}%
	}
\end{table}

\begin{table}[htbp]
	\centering
	\caption{Relative bias for large coefficients and estimated sparsity for $p = 128$, $s = 10$, $\sigma = 5\times10^{-2}$. Sparsity values are shown as mean (SD) in decimal format.}
	\label{tab:bias_p128_s10_s5e-2}
	\scriptsize
	\renewcommand{\arraystretch}{1.2}
	\resizebox{\linewidth}{!}{%
		\begin{tabular}{c|c|ccccccc}
			\toprule
			\multirow{2}{*}{$n/p$} & \multirow{2}{*}{Metric} & \multicolumn{7}{c}{Penalty} \\
			\cmidrule{3-9}
			& & $\ell_1$ & $\ell_{1/2}$ & SCAD & MCP & TL1 & LOG & EXP \\
			\midrule
			0.1 & Bias & 9.23e-1 (8.60e-2) & 9.19e-1 (9.02e-2) & 9.23e-1 (8.99e-2) & 9.24e-1 (9.20e-2) & 9.30e-1 (9.24e-2) & 9.32e-1 (9.07e-2) & 9.26e-1 (9.06e-2) \\
			& Sparsity & 60.74 (5.87) & 31.00 (8.61) & 25.78 (8.62) & 24.25 (10.10) & 18.67 (3.70) & 8.24 (2.30) & 20.33 (4.11) \\
			\midrule
			0.2 & Bias & 8.09e-1 (1.63e-1) & 8.11e-1 (1.54e-1) & 8.17e-1 (1.70e-1) & 8.29e-1 (1.61e-1) & 7.53e-1 (3.24e-1) & 7.95e-1 (2.32e-1) & 7.90e-1 (2.09e-1) \\
			& Sparsity & 68.64 (5.95) & 39.93 (6.47) & 30.84 (8.95) & 30.26 (11.67) & 19.17 (4.98) & 12.40 (2.79) & 25.10 (5.18) \\
			\midrule
			0.3 & Bias & 6.86e-1 (2.78e-1) & 7.01e-1 (2.46e-1) & 6.99e-1 (2.90e-1) & 7.18e-1 (2.66e-1) & 5.16e-1 (4.51e-1) & 6.10e-1 (3.82e-1) & 6.29e-1 (3.63e-1) \\
			& Sparsity & 69.64 (5.94) & 39.27 (8.99) & 30.59 (11.74) & 28.08 (12.97) & 18.25 (4.17) & 13.14 (4.17) & 23.86 (5.77) \\
\bottomrule
		\end{tabular}%
	}
\end{table}


\begin{table}[htbp]
	\centering
	\caption{\text{RelErr} for $p = 256$, $s = 15$, $\sigma = 10^{-2}$.}
	\label{tab:error_p256_s15_s1e-2}
	\scriptsize
	\renewcommand{\arraystretch}{1.2}
	\resizebox{\linewidth}{!}{%
		\begin{tabular}{c|ccccccc}
			\toprule
			$n/p$ & $\ell_1$ & $\ell_{1/2}$ & SCAD & MCP & TL1 & LOG & EXP \\
			\midrule
			0.1 & 1.12 (6.24e-2) & 1.13 (6.09e-2) & 1.13 (6.37e-2) & 1.12 (6.08e-2) & 1.23 (9.68e-2) & 1.24 (9.26e-2) & 1.18 (7.26e-2) \\
			0.2 & 1.08 (1.51e-1) & 1.12 (1.16e-1) & 1.08 (1.40e-1) & 1.08 (1.31e-1) & 1.06 (4.76e-1) & 1.15 (3.18e-1) & 1.17 (2.78e-1) \\
			0.3 & 9.11e-1 (3.71e-1) & 1.06 (1.92e-1) & 9.29e-1 (3.28e-1) & 9.85e-1 (2.04e-1) & 7.17e-1 (6.34e-1) & 8.12e-1 (5.92e-1) & 9.56e-1 (5.45e-1) \\
			0.4 & 5.03e-1 (5.84e-1) & 8.36e-1 (4.64e-1) & 5.40e-1 (5.23e-1) & 5.93e-1 (4.97e-1) & 3.72e-1 (5.86e-1) & 4.32e-1 (5.91e-1) & 5.23e-1 (6.35e-1) \\
			0.5 & 1.81e-1 (4.36e-1) & 4.47e-1 (5.46e-1) & 2.32e-1 (4.34e-1) & 2.70e-1 (4.53e-1) & 1.01e-1 (3.45e-1) & 1.64e-1 (4.25e-1) & 2.27e-1 (4.91e-1) \\
			0.6 & 2.72e-2 (1.91e-1) & 1.00e-1 (3.26e-1) & 5.39e-2 (2.37e-1) & 5.50e-2 (2.42e-1) & 3.68e-2 (2.10e-1) & 4.99e-2 (2.46e-1) & 4.09e-2 (2.34e-1) \\
			0.7 & 6.03e-5 (1.42e-5) & 2.60e-4 (5.49e-4) & 2.91e-5 (8.28e-6) & 2.75e-5 (8.16e-6) & 1.29e-2 (1.28e-1) & 6.54e-5 (3.34e-5) & 3.70e-5 (1.09e-5) \\
			0.8 & 5.76e-5 (1.53e-5) & 1.27e-2 (1.26e-1) & 9.88e-3 (9.85e-2) & 2.68e-5 (8.25e-6) & 4.29e-5 (1.31e-5) & 5.77e-5 (2.60e-5) & 3.63e-5 (1.15e-5) \\
			0.9 & 5.06e-5 (1.32e-5) & 1.35e-2 (1.34e-1) & 2.49e-5 (7.45e-6) & 2.35e-5 (7.39e-6) & 3.76e-5 (1.15e-5) & 4.90e-5 (2.12e-5) & 3.17e-5 (1.04e-5) \\
			1.0 & 5.11e-5 (1.34e-5) & 5.39e-5 (1.32e-4) & 2.56e-5 (7.37e-6) & 2.41e-5 (7.23e-6) & 3.88e-5 (1.19e-5) & 5.00e-5 (2.12e-5) & 3.24e-5 (1.04e-5) \\
			\bottomrule
		\end{tabular}%
	}
\end{table}

\begin{table}[htbp]
	\centering
	\caption{Computation time (seconds) for $p = 256$, $s = 15$, $\sigma = 10^{-2}$.}
	\label{tab:time_p256_s15_s1e-2}
	\scriptsize
	\renewcommand{\arraystretch}{1.2}
	\resizebox{\linewidth}{!}{%
		\begin{tabular}{c|ccccccc}
			\toprule
			$n/p$& $\ell_1$ & $\ell_{1/2}$ & SCAD & MCP & TL1 & LOG & EXP \\
			\midrule
			0.1 & 2.83e-2 (4.57e-3) & 8.32e-2 (3.11e-2) & 1.12e-2 (7.39e-3) & 9.67e-3 (2.31e-3) & 4.12e-1 (5.96e-2) & 1.87e-1 (1.84e-2) & 4.55e-1 (6.96e-2) \\
			0.2 & 1.38e-1 (3.50e-2) & 3.45e-1 (8.57e-2) & 2.10e-2 (5.82e-3) & 2.01e-2 (3.66e-3) & 6.05e-1 (1.55e-1) & 4.12e-1 (5.86e-2) & 1.43e+0 (7.63e-1) \\
			0.3 & 1.10e+0 (3.54e-1) & 9.75e-1 (4.93e-1) & 5.93e-2 (2.50e-2) & 5.31e-2 (1.34e-2) & 8.77e-1 (6.03e-1) & 9.61e-1 (5.09e-1) & 5.81e+0 (3.14e+0) \\
			0.4 & 1.91e+0 (2.32e-1) & 7.56e-1 (4.47e-1) & 1.40e-1 (6.83e-2) & 1.23e-1 (4.95e-2) & 7.00e-1 (6.44e-1) & 1.02e+0 (8.27e-1) & 6.40e+0 (6.07e+0) \\
			0.5 & 1.71e+0 (4.46e-1) & 7.54e-1 (6.17e-1) & 1.85e-1 (5.06e-2) & 1.57e-1 (3.69e-2) & 4.41e-1 (4.69e-1) & 6.03e-1 (7.24e-1) & 4.08e+0 (6.30e+0) \\
			0.6 & 1.21e+0 (3.07e-1) & 4.74e-1 (4.85e-1) & 2.19e-1 (4.00e-2) & 1.92e-1 (3.42e-2) & 3.48e-1 (2.57e-1) & 4.46e-1 (6.40e-1) & 1.39e+0 (3.37e+0) \\
			0.7 & 7.13e-1 (1.22e-1) & 2.51e-1 (1.56e-1) & 1.85e-1 (3.43e-2) & 1.64e-1 (3.12e-2) & 2.47e-1 (7.83e-2) & 3.08e-1 (4.34e-1) & 6.83e-1 (3.39e-1) \\
			0.8 & 6.36e-1 (8.26e-2) & 2.15e-1 (5.25e-2) & 1.95e-1 (2.61e-2) & 1.81e-1 (6.40e-2) & 2.50e-1 (3.02e-2) & 3.97e-1 (5.71e-1) & 5.95e-1 (1.85e-1) \\
			0.9 & 5.78e-1 (7.67e-2) & 2.47e-1 (1.90e-1) & 2.13e-1 (4.46e-2) & 2.21e-1 (2.83e-1) & 2.61e-1 (1.34e-1) & 4.32e-1 (6.97e-1) & 5.89e-1 (1.92e-1) \\
			1.0 & 5.57e-1 (6.77e-2) & 2.30e-1 (8.07e-2) & 2.58e-1 (3.42e-1) & 2.32e-1 (1.56e-1) & 2.64e-1 (3.26e-2) & 3.99e-1 (6.68e-1) & 5.84e-1 (2.51e-1) \\
			\bottomrule
		\end{tabular}%
	}
\end{table}

\begin{table}[htbp]
	\centering
	\caption{Support recovery metrics for $p = 256$, $s = 15$, $\sigma = 10^{-2}$.}
	\label{tab:support_p256_s15_s1e-2}
	\scriptsize
	\renewcommand{\arraystretch}{1.2}
	\resizebox{\linewidth}{!}{%
		\begin{tabular}{c|c|ccccccc}
			\toprule
			\multirow{2}{*}{$n/p$} & \multirow{2}{*}{Metric} & \multicolumn{7}{c}{Penalty} \\
			\cmidrule{3-9}
			& & $\ell_1$ & $\ell_{1/2}$ & SCAD & MCP & TL1 & LOG & EXP \\
			\midrule
			0.1 & TPR & 5.43e-1 (1.30e-1) & 3.29e-1 (1.27e-1) & 3.10e-1 (1.47e-1) & 3.16e-1 (1.64e-1) & 1.67e-1 (9.64e-2) & 9.87e-2 (7.78e-2) & 2.03e-1 (1.08e-1) \\
			& FPR & 4.70e-1 (3.11e-2) & 2.32e-1 (2.98e-2) & 2.26e-1 (7.33e-2) & 2.29e-1 (8.89e-2) & 1.28e-1 (2.65e-2) & 5.53e-2 (1.09e-2) & 1.57e-1 (2.76e-2) \\
			& F1  & 1.20e-1 (2.81e-2) & 1.29e-1 (4.61e-2) & 1.24e-1 (4.78e-2) & 1.26e-1 (5.65e-2) & 1.02e-1 (5.66e-2) & 9.92e-2 (7.81e-2) & 1.09e-1 (5.78e-2) \\
			\midrule
			0.2 & TPR & 5.95e-1 (1.22e-1) & 4.29e-1 (1.24e-1) & 3.94e-1 (1.53e-1) & 3.91e-1 (1.41e-1) & 2.95e-1 (3.16e-1) & 1.91e-1 (1.62e-1) & 2.81e-1 (1.75e-1) \\
			& FPR & 4.99e-1 (3.64e-2) & 2.97e-1 (3.87e-2) & 2.76e-1 (7.31e-2) & 2.76e-1 (7.54e-2) & 1.02e-1 (3.15e-2) & 7.71e-2 (2.65e-2) & 1.56e-1 (3.89e-2) \\
			& F1  & 1.24e-1 (2.59e-2) & 1.39e-1 (4.13e-2) & 1.37e-1 (5.33e-2) & 1.36e-1 (4.89e-2) & 2.04e-1 (2.22e-1) & 1.70e-1 (1.84e-1) & 1.57e-1 (1.33e-1) \\
\bottomrule
		\end{tabular}%
	}
\end{table}

\begin{table}[htbp]
	\centering
	\caption{Relative bias for large coefficients and estimated sparsity for $p = 256$, $s = 15$, $\sigma = 10^{-2}$. Sparsity values are shown as mean (SD) in decimal format.}
	\label{tab:bias_p256_s15_s1e-2}
	\scriptsize
	\renewcommand{\arraystretch}{1.2}
	\resizebox{\linewidth}{!}{%
		\begin{tabular}{c|c|ccccccc}
			\toprule
			\multirow{2}{*}{$n/p$} & \multirow{2}{*}{Metric} & \multicolumn{7}{c}{Penalty} \\
			\cmidrule{3-9}
			& & $\ell_1$ & $\ell_{1/2}$ & SCAD & MCP & TL1 & LOG & EXP \\
			\midrule
			0.1 & Bias & 9.27e-1 (7.17e-2) & 9.27e-1 (7.49e-2) & 9.28e-1 (7.12e-2) & 9.29e-1 (6.86e-2) & 9.35e-1 (7.59e-2) & 9.35e-1 (7.77e-2) & 9.24e-1 (7.61e-2) \\
			& Sparsity & 121.40 (7.71) & 60.82 (7.93) & 59.14 (18.99) & 60.02 (22.86) & 33.25 (6.82) & 14.80 (2.50) & 40.85 (6.86) \\
			\midrule
			0.2 & Bias & 8.43e-1 (1.45e-1) & 8.57e-1 (1.25e-1) & 8.60e-1 (1.34e-1) & 8.69e-1 (1.24e-1) & 7.66e-1 (3.43e-1) & 8.21e-1 (2.37e-1) & 8.27e-1 (2.14e-1) \\
			& Sparsity & 129.20 (8.81) & 77.96 (9.47) & 72.43 (18.58) & 72.28 (19.10) & 28.99 (6.05) & 21.43 (5.02) & 41.92 (8.51) \\
			\midrule
			0.3 & Bias & 6.63e-1 (2.95e-1) & 7.47e-1 (1.84e-1) & 7.22e-1 (2.74e-1) & 7.71e-1 (1.84e-1) & 5.09e-1 (4.52e-1) & 5.71e-1 (4.24e-1) & 6.32e-1 (3.72e-1) \\
			& Sparsity & 128.70 (9.80) & 79.35 (12.44) & 79.08 (28.29) & 82.03 (25.79) & 26.97 (5.01) & 22.26 (7.48) & 37.10 (9.89) \\
\bottomrule
		\end{tabular}%
	}
\end{table}


\begin{table}[htbp]
	\centering
	\caption{\text{RelErr} for $p = 256$, $s = 15$, $\sigma = 5\times10^{-2}$.}
	\label{tab:error_p256_s15_s5e-2}
	\scriptsize
	\renewcommand{\arraystretch}{1.2}
	\resizebox{\linewidth}{!}{%
		\begin{tabular}{c|ccccccc}
			\toprule
			$n/p$ & $\ell_1$ & $\ell_{1/2}$ & SCAD & MCP & TL1 & LOG & EXP \\
			\midrule
			0.1 & 1.12 (6.52e-2) & 1.13 (6.32e-2) & 1.14 (6.70e-2) & 1.13 (6.50e-2) & 1.24 (9.17e-2) & 1.25 (8.16e-2) & 1.19 (7.27e-2) \\
			0.2 & 1.09 (1.48e-1) & 1.11 (1.12e-1) & 1.07 (1.46e-1) & 1.06 (1.43e-1) & 1.05 (4.94e-1) & 1.15 (3.44e-1) & 1.17 (2.89e-1) \\
			0.3 & 9.63e-1 (3.46e-1) & 1.10 (1.73e-1) & 9.66e-1 (3.08e-1) & 9.84e-1 (2.39e-1) & 7.46e-1 (6.31e-1) & 9.24e-1 (5.36e-1) & 1.02e+0 (4.89e-1) \\
			0.4 & 3.56e-1 (5.27e-1) & 7.62e-1 (4.56e-1) & 5.00e-1 (5.02e-1) & 5.13e-1 (4.77e-1) & 2.39e-1 (4.97e-1) & 3.31e-1 (5.47e-1) & 3.59e-1 (5.69e-1) \\
			0.5 & 2.39e-1 (4.97e-1) & 4.03e-1 (5.29e-1) & 2.50e-1 (4.57e-1) & 2.77e-1 (4.70e-1) & 1.87e-1 (4.48e-1) & 2.21e-1 (4.72e-1) & 2.23e-1 (4.81e-1) \\
			0.6 & 5.22e-2 (2.56e-1) & 1.25e-1 (3.62e-1) & 5.54e-2 (2.43e-1) & 8.20e-2 (2.83e-1) & 3.67e-2 (2.09e-1) & 4.96e-2 (2.42e-1) & 7.52e-2 (2.99e-1) \\
			0.7 & 3.12e-4 (8.31e-5) & 1.24e-2 (1.21e-1) & 1.51e-4 (4.34e-5) & 1.43e-4 (4.19e-5) & 2.26e-4 (7.07e-5) & 3.20e-4 (1.78e-4) & 1.90e-4 (5.96e-5) \\
			0.8 & 2.93e-4 (8.83e-5) & 2.25e-4 (3.16e-4) & 1.47e-4 (4.72e-5) & 1.40e-4 (4.58e-5) & 2.20e-4 (7.83e-5) & 2.98e-4 (1.81e-4) & 1.86e-4 (6.90e-5) \\
			0.9 & 2.51e-4 (6.59e-5) & 1.87e-4 (2.55e-4) & 1.29e-4 (3.65e-5) & 1.22e-4 (3.61e-5) & 1.90e-4 (5.41e-5) & 2.44e-4 (9.08e-5) & 1.60e-4 (4.68e-5) \\
			1.0 & 2.44e-4 (5.76e-5) & 1.76e-4 (2.94e-4) & 1.23e-4 (3.08e-5) & 1.16e-4 (2.98e-5) & 1.83e-4 (4.58e-5) & 2.39e-4 (8.61e-5) & 1.55e-4 (3.97e-5) \\
			\bottomrule
		\end{tabular}%
	}
\end{table}

\begin{table}[htbp]
	\centering
	\caption{Computation time (seconds) for $p = 256$, $s = 15$, $\sigma = 5\times10^{-2}$.}
	\label{tab:time_p256_s15_s5e-2}
	\scriptsize
	\renewcommand{\arraystretch}{1.2}
	\resizebox{\linewidth}{!}{%
		\begin{tabular}{c|ccccccc}
			\toprule
			$n/p$& $\ell_1$ & $\ell_{1/2}$ & SCAD & MCP & TL1 & LOG & EXP \\
			\midrule
			0.1 & 3.27e-2 (1.44e-2) & 9.40e-2 (5.38e-2) & 1.15e-2 (5.52e-3) & 1.08e-2 (3.77e-3) & 4.79e-1 (1.85e-1) & 2.21e-1 (9.09e-2) & 4.66e-1 (1.17e-1) \\
			0.2 & 1.44e-1 (4.13e-2) & 3.39e-1 (9.34e-2) & 2.09e-2 (6.84e-3) & 2.05e-2 (4.85e-3) & 6.14e-1 (1.43e-1) & 4.18e-1 (6.53e-2) & 1.48e+0 (7.28e-1) \\
			0.3 & 1.18e+0 (3.47e-1) & 1.02e+0 (4.85e-1) & 6.69e-2 (3.26e-2) & 6.11e-2 (1.93e-2) & 9.50e-1 (6.22e-1) & 1.22e+0 (4.41e-1) & 6.38e+0 (3.19e+0) \\
			0.4 & 1.72e+0 (3.60e-1) & 6.61e-1 (3.55e-1) & 1.15e-1 (4.68e-2) & 1.07e-1 (4.12e-2) & 4.88e-1 (4.80e-1) & 8.11e-1 (7.11e-1) & 5.17e+0 (5.75e+0) \\
			0.5 & 2.19e+0 (9.32e-1) & 7.95e-1 (5.98e-1) & 2.17e-1 (1.05e-1) & 1.90e-1 (9.46e-2) & 7.47e-1 (8.75e-1) & 9.61e-1 (1.11e+0) & 4.36e+0 (6.81e+0) \\
			0.6 & 1.42e+0 (5.04e-1) & 4.91e-1 (5.19e-1) & 2.04e-1 (3.69e-2) & 1.78e-1 (2.89e-2) & 3.48e-1 (2.20e-1) & 5.65e-1 (7.40e-1) & 2.03e+0 (4.83e+0) \\
			0.7 & 9.24e-1 (2.57e-1) & 2.80e-1 (1.79e-1) & 2.08e-1 (6.04e-2) & 1.76e-1 (4.52e-2) & 2.86e-1 (7.15e-2) & 3.26e-1 (2.15e-1) & 6.83e-1 (1.62e-1) \\
			0.8 & 6.14e-1 (9.35e-2) & 2.02e-1 (4.70e-2) & 1.69e-1 (2.24e-2) & 1.54e-1 (3.15e-2) & 2.57e-1 (2.24e-1) & 3.54e-1 (4.31e-1) & 5.97e-1 (1.65e-1) \\
			0.9 & 5.99e-1 (8.00e-2) & 2.08e-1 (5.02e-2) & 1.89e-1 (2.60e-2) & 1.70e-1 (2.13e-2) & 2.52e-1 (2.61e-2) & 2.84e-1 (2.20e-1) & 6.36e-1 (4.76e-1) \\
			1.0 & 5.88e-1 (8.52e-2) & 2.22e-1 (7.01e-2) & 2.17e-1 (4.26e-2) & 1.99e-1 (6.94e-2) & 3.39e-1 (4.94e-1) & 3.26e-1 (3.61e-1) & 5.92e-1 (4.09e-1) \\
			\bottomrule
		\end{tabular}%
	}
\end{table}

\begin{table}[htbp]
	\centering
	\caption{Support recovery metrics for $p = 256$, $s = 15$, $\sigma = 5\times10^{-2}$.}
	\label{tab:support_p256_s15_s5e-2}
	\scriptsize
	\renewcommand{\arraystretch}{1.2}
	\resizebox{\linewidth}{!}{%
		\begin{tabular}{c|c|ccccccc}
			\toprule
			\multirow{2}{*}{$n/p$} & \multirow{2}{*}{Metric} & \multicolumn{7}{c}{Penalty} \\
			\cmidrule{3-9}
			& & $\ell_1$ & $\ell_{1/2}$ & SCAD & MCP & TL1 & LOG & EXP \\
			\midrule
			0.1 & TPR & 5.17e-1 (1.27e-1) & 3.19e-1 (1.09e-1) & 2.99e-1 (1.29e-1) & 3.03e-1 (1.62e-1) & 1.75e-1 (9.51e-2) & 1.02e-1 (7.84e-2) & 2.35e-1 (1.10e-1) \\
			& FPR & 4.78e-1 (2.99e-2) & 2.33e-1 (3.35e-2) & 2.29e-1 (6.66e-2) & 2.27e-1 (8.65e-2) & 1.27e-1 (2.33e-2) & 5.58e-2 (1.10e-2) & 1.55e-1 (2.63e-2) \\
			& F1  & 1.12e-1 (2.61e-2) & 1.26e-1 (3.92e-2) & 1.20e-1 (4.38e-2) & 1.20e-1 (5.11e-2) & 1.08e-1 (5.66e-2) & 1.01e-1 (7.68e-2) & 1.26e-1 (5.80e-2) \\
			\midrule
			0.2 & TPR & 6.09e-1 (1.38e-1) & 4.51e-1 (1.45e-1) & 4.40e-1 (1.61e-1) & 4.31e-1 (1.62e-1) & 3.09e-1 (3.26e-1) & 2.07e-1 (1.95e-1) & 3.04e-1 (1.86e-1) \\
			& FPR & 5.02e-1 (3.94e-2) & 2.97e-1 (4.35e-2) & 2.83e-1 (8.28e-2) & 2.78e-1 (8.57e-2) & 1.01e-1 (3.24e-2) & 7.56e-2 (2.75e-2) & 1.51e-1 (3.67e-2) \\
			& F1  & 1.26e-1 (2.72e-2) & 1.46e-1 (4.55e-2) & 1.51e-1 (6.02e-2) & 1.50e-1 (6.21e-2) & 2.18e-1 (2.40e-1) & 1.86e-1 (2.11e-1) & 1.72e-1 (1.41e-1) \\
\bottomrule
		\end{tabular}%
	}
\end{table}

\begin{table}[htbp]
	\centering
	\caption{Relative bias for large coefficients and estimated sparsity for $p = 256$, $s = 15$, $\sigma = 5\times10^{-2}$. Sparsity values are shown as mean (SD) in decimal format.}
	\label{tab:bias_p256_s15_s5e-2}
	\scriptsize
	\renewcommand{\arraystretch}{1.2}
	\resizebox{\linewidth}{!}{%
		\begin{tabular}{c|c|ccccccc}
			\toprule
			\multirow{2}{*}{$n/p$} & \multirow{2}{*}{Metric} & \multicolumn{7}{c}{Penalty} \\
			\cmidrule{3-9}
			& & $\ell_1$ & $\ell_{1/2}$ & SCAD & MCP & TL1 & LOG & EXP \\
			\midrule
			0.1 & Bias & 9.40e-1 (7.55e-2) & 9.38e-1 (8.11e-2) & 9.43e-1 (7.66e-2) & 9.42e-1 (7.58e-2) & 9.37e-1 (7.23e-2) & 9.45e-1 (7.46e-2) & 9.34e-1 (7.59e-2) \\
			& Sparsity & 122.90 (7.76) & 60.95 (8.67) & 59.62 (17.00) & 59.13 (22.42) & 33.16 (5.91) & 14.98 (2.70) & 40.99 (6.56) \\
			\midrule
			0.2 & Bias & 8.60e-1 (1.92e-1) & 8.64e-1 (1.66e-1) & 8.65e-1 (1.77e-1) & 8.64e-1 (1.68e-1) & 7.46e-1 (3.53e-1) & 8.25e-1 (2.72e-1) & 8.14e-1 (2.29e-1) \\
			& Sparsity & 130.10 (9.88) & 78.23 (10.89) & 74.69 (20.73) & 73.50 (21.51) & 29.07 (5.22) & 21.32 (4.95) & 40.88 (7.70) \\
			\midrule
			0.3 & Bias & 7.07e-1 (2.74e-1) & 7.93e-1 (1.65e-1) & 7.56e-1 (2.50e-1) & 7.74e-1 (2.02e-1) & 5.37e-1 (4.59e-1) & 6.61e-1 (3.84e-1) & 6.81e-1 (3.45e-1) \\
			& Sparsity & 129.50 (8.71) & 79.61 (10.50) & 76.36 (27.20) & 76.22 (27.41) & 27.96 (4.92) & 21.78 (5.99) & 37.99 (9.15) \\
\bottomrule
		\end{tabular}%
	}
\end{table}

	
	
	\bibliography{mybibfile}
	
\end{document}